\UseRawInputEncoding
\documentclass[12pt]{amsart}
\usepackage{amssymb}
\usepackage{amsfonts}
\usepackage{amsmath}
\usepackage{graphicx}%
\usepackage{xcolor}%
\usepackage{tikz}
\usepackage[hidelinks]{hyperref}

\newtheorem{theorem}{Theorem}

\newtheorem{definition}[theorem]{Definition}
\newtheorem{corollary}[theorem]{Corollary}

\newtheorem{proposition}[theorem]{Proposition}
\theoremstyle{definition}
\newtheorem{remark}[theorem]{Remark}

\def \mb{\mathbb}

\def \bf{\mathbf}

\def \Z{\mb Z}                  
\def \R{\mb R}                 

\def \a{\alpha}         
\def \b{\beta}           
\def \D{\Delta}         
\def \vp{\varphi}       
\def \th{\theta}       

\newcommand {\diag} {\text{diag}}

\def \S{\mb S}        
\def \H{\mb H}        
\def\v{{\bf v}}
\def\u{{\bf u}}

\def\c{{\bf c}}

\newcommand {\q} {\mathbf{q}}

\title{Compactness and index of  ordinary central configurations  for  the curved $N$-body problem
}
\begin{document}
	\maketitle
	\markboth{Shuqiang Zhu}{  Compactness and index of OCC  for  the curved $n$-body problem}
	\vspace{-0.5cm}
	\author       
	\bigskip
	\begin{center}
		{\emph{In memoriam  of Florin Diacu}}\\
		{Shuqiang Zhu}\\
{\footnotesize 
	
	School of Economic and Mathematics,  Southwestern University of Finance and Economics, \\
	
	Chengdu 611130, China\\
	
	zhusq@swufe.edu.cn 
	
}

\end{center}

	\begin{abstract}	For the curved $n$-body problem, we show that  the set of ordinary central configurations is away from most singular configurations in $\H^3$, and away from a subset of  singular configurations in $\S^3$. We also show that 
			each of the  $n!/2$ geodesic  ordinary central configurations  for  $n$ masses has Morse index $n-2$. 
Then we get a direct corollary that there are at least $\frac{(3n-4)(n-1)!}{2}$	ordinary central configurations  for given $n$ masses if all  ordinary central configurations of these masses are non-degenerate.  \end{abstract}

    \vspace{2mm}
    
    \vspace{8mm}
    
    \section{introduction}

 The curved $n$-body problem studies the motion of particles interacting under the cotangent potential  in 3-dimensional sphere and 3-dimensional hyperbolic sphere.  It is a natural extension of the   Newtonian $n$-body problem in $\R^3$.  It roots in the research of Bolyai and Lobachevsky. For history and  recent advances,  one can refer to  Arnold et al. \cite{AKN06},  Borisov et al. \cite{BMK04} and Diacu \cite{Dia13-1}.  There are many researches in this area over the past two decades on the Kepler problem, two-body problem,  relative  equilibria, stability of periodic orbits,  etc.  
       
 The curved $n$-body problem is a Lagrangian mechanical  system. 
 A solution of the Euler-Lagrange's equation in the form of $A(t)\q$ is called a  relative  equilibrium if  $A(t)$   is a 1-parameter subgroup  of the symmetry group. 
  The topic of   relative  equilibria has received much attention recently (see  \cite{BGMM18,  Dia13-1,DPV12,DS14, Kil98, MS17, PS19, Tib14, ZZ16}  among others).  
  Unlike that of the   Newtonian $n$-body problem in $\R^3$,  the set of   relative  equilibria is divided into five classes, see Section \ref{sec:cnbpre}.  
    Diacu, Stoica and Zhu  introduce  a unified criterion for  relative  equilibria of the  curved $n$-body problem. The configurations of  relative  equilibrium are characterized as critical points of a single function \cite{DSZ18}. The criterion simplifies  the job of  finding   relative  equilibria \cite{DZ20}.  The configurations  are called central configurations and the set central configurations is divided into the ordinary ones and the special ones.    A counting problem of  ordinary central configurations is also proposed, see  Section 2.4.

In the   Newtonian $n$-body problem, the celebrated  problem of the finiteness of    configurations of relative equilibrium is still unsolved for $n>5$ up to now  \cite{AK12, Sma70-3, Win47}. On the other hand,  the collinear case is clear.  For any $n$ masses, there are exactly  $n!/2$ collinear  configurations of relative equilibrium  (Moulton \cite{Mou10}) and  their Morse index is $n-2$ \cite{Pac87}.  In general the set of normalized configurations of relative equilibrium  is known to be compact (Shub \cite{Shu70}) in the configuration space. 
Moreover, Palmore \cite{Pal73} proved that there are at least $\frac{(3n-4)(n-1)!}{2}$
 configurations of relative equilibrium  for given $n$ masses if all  configurations of relative equilibrium  of these masses are non-degenerate.

The purpose of  this paper is to  extend the results mentioned  above to the curved $n$-body problem. More precisely,  we show that  the set of ordinary central configurations is away from most singular configurations in $\H^3$, and away from a subset of  singular configurations in $\S^3$. We show that  there are $n!/2$ ``collinear''  ordinary central configurations and their Morse index is $n-2$ in $\H^3$, which also holds  in $\S^3$ provided that some conditions are satisfied. 
Furthermore,  Palmore's estimation 
also holds.

   The paper is organized as follows. In Section \ref{sec:cnbp}, we   review  the $n$-body problem in the three manifolds, $\R^3$, $\S^3$ and $\H^3$,  the criterion for  relative  equilibria  and the counting of ordinary central configurations.  
  Then we state our main results.  In Section \ref{sec:com}, we prove the results on 
   compactness of the set of   ordinary central configurations.  In Section \ref{sec:ind}, we prove the results on  
   the number and Morse index of the geodesic  ordinary central configurations.  In Appendix, we  discuss association between   the central configurations and  the relative  equilibria motions.

    \section{Relative equilibria   of the curved $n$-body problem and  main results  }\label{sec:cnbp}
    
     In this section,  we first briefly review the  $n$-body problem  on the  three manifolds, $\R^{3}$, $\S^3$ and $\H^3$,  criterion for  relative  equilibria,  then state the main results of this paper. 
     Vectors are all column vectors, but written as row vectors in the  text.  The masses $m_1, ..., m_n$ are always positive. 
     
     \subsection{Relative equilibria of the  $n$-body  problem in $\R^3$}
     
     The   Newtonian $n$-body  problem in $\R^3$  studies the motion of $n$ particles in $R^3$  with masses $m_1, ..., m_n$ under the gravitational interaction. It   is a Lagrange mechanical system with Lagrangian function
     \[  L=  \frac{1}{2} \sum_{i=1}^n m_i \dot{\q}_i \cdot\dot{\q}_i   -U_0(\q),   \]
     where $\q=(\q_1, ..., \q_n)$, $\q_i=(x_i, y_i, z_i)\in \R^3$,  and $U_0=\sum \frac{m_im_j}{||\q_i-\q_j||} $ is the potential defined on the configuration space $(\R^3)^n-\D$, $\D=\cup_{1\le i<j\le n}\{\q\in (\R^3)^n\ \! | \! \q_i=\q_j\}$ 
     
     A \emph{ relative  equilibrium}  is  an integral curve of the system in the form of $A(t)\q$, with  $A(t)$  being a  uniform rotation in $SO(3)$. The corresponding configuration $\q$ is called \emph{a configuration of relative  equilibrium}, 
      Every uniform rotation in $SO(3)$ has an fixed axis. Assume that  the $z$-axis is the rotation axis. Then it is well-known that the  configurations of relative  equilibrium  must be on a plane perpendicular to the  $z$-axis.   Without lose of generality,  we assume that they are on the plane $\{z=0\}$, then they   are  critical points of 
    $ U_0-\lambda I_0$ for some $\lambda\in \R$,  where  $I_0(\q)=\sum m_i (x_i^2+ y^2_i)$. 
      
     Two  configurations of relative equilibrium  are said to be in one class if one can be deduced from the other by some rotation in  $SO(2)$  and  non-zero scalar multiplication. 
     The finiteness problem on  configurations of relative equilibrium  is: given $n$ masses $m_1, ..., m_n$, is the number of classes of  configurations of relative equilibrium  finite?
   In other words,  
is  the number of  configurations of relative equilibrium  in $ \{ \q\in(\R^2)^n-\D \! | \ \! I_0(\q)=1 \}/S^1 $ finite?
      For history and advance of this problem, we refer the readers to \cite{AK12, Sma70-3, Win47} and the references therein. 
      
     \subsection{ The  curved $n$-body  problem in $\S^3$ and $\H^3$}\label{sec:cnbpeuler}
  The  curved  $n$-body  problem   studies  the motion of $n$ particles interacting  under the so-called cotangent potential  in $\S^3$  and $\H^3$.   The two manifolds can be parameterized in many ways.  The Cartesian coordinates are convenient in many cases.  That is,  $\S^3$ (resp. $\H^3$) is the unit sphere in $\R^4$ (resp. $\R^{3,1}$). 
 Recall that the`inner product in those 4-dimensional linear space are 
$ \q_1\cdot \q_2 = x_1x_2+y_1y_2+z_1z_2 +\sigma w_1w_2,  $
  where $\q_i=(x_i,y_i,z_i,w_i)$,  $\sigma=1$ for $\R^4$ and $\sigma=-1$ for $\R^{3,1}$.   Then $\S^3=\{ (x,y,z,w)\in \R^4| x^2+y^2+z^2+w^2=1  \}$,  $\H^3=\{(x,y,z,w)\in \R^{3,1}| x^2+y^2+z^2-w^2=-1, w\ge 1  \}$.  The Riemannian metrics on  $\S^3$ and $\H^3$ are induced from the ``inner product''. The  distance   between two   point masses $m_i$ and $m_j$,   $d_{ij}=d(\q_i , \q_j)$,  is   computed by
  $\cos  d_{ij}= \q_i \cdot \q_j$  on $\S^3$ and   $\cosh  d_{ij}(\q)= -\q_i \cdot \q_j$  on $\H^3$.

The curved $n$-body problem in $\S^3$ is a Lagrange mechanical system with Lagrangian function
\[  L=  \frac{1}{2} \sum_{i=1}^n m_i \dot{\q}_i \cdot\dot{\q}_i   -U_1(\q),   \]
  where $\q=(\q_1, ..., \q_n)$, $\q_i\in \S^3$.  Let $\D_+$ be the 
  singular set   $\D_+=\cup_{1\le i<j\le n}\{\q\in (\S^3)^n\ \! | \!  \q_i=\pm \q_j\}$. 
   The potential  is $U_1=\sum m_im_j \cot d_{ij}$  defined on the configuration space $(\S^3)^n-\D_+$.   The  equations of motion  are
\cite{Dia13-1, DSZ18}:  
\begin{equation}\label{equ:motion}
\begin{cases}
m_i\ddot\q_i=\sum_{j=1,j\ne i}^n\frac{m_im_j [\q_j-\cos d_{ij}\q_i]}{\sin^3 d_{ij}}-\sigma m_i(\dot{\q}_i\cdot \dot{\q}_i)\q_i\cr
\q_i\cdot \q_i=\sigma, \ \ \ \ i=1,..., n.
\end{cases}
\end{equation}
 Recall that $\sigma=1$ for $\S^3$ and $\sigma=-1$ for $\H^{3}$. 

Likewise, the curved $n$-body problem in $\H^3$ is a Lagrange mechanical system with Lagrangian function
\[  L=  \frac{1}{2} \sum_{i=1}^n m_i \dot{\q}_i \cdot\dot{\q}_i   -U_{-1}(\q),   \]
where $\q=(\q_1, ..., \q_n)$, $\q_i\in \H^3$.  Let $\D_-$ be the 
singular set 
$\D_-=\cup_{1\le i<j\le n}\{\q\in (\H^3)^n\ \! | \! \q_i=\q_j\}$.  The potential is  $U_{-1}(\q)=\sum m_im_j \coth d_{ij}$ defined  on the configuration space $(\H^3)^n-\D_-$.  Replacing the trigonometrical functions  by the hyperbolic ones and putting $\sigma=-1$,   equations  \eqref{equ:motion}  become the equations of motion  for the curved $n$-body problem in $\H^3$. 
    \subsection{Relative equilibria   in $\S^3$ and $\H^3$}\label{sec:cnbpre}
   A {simple mechanical system with symmetry} in the terminology of Smale 
    is a Lagrange mechanical system on a manifold $M$ in the form of 
    $L=K(\dot{\q}) +U(\q)$, where $K$ is a Riemannian metric on $M$ and there is a Lie group $G$ acting on $M$ preserving  $K$ and $U$ smoothly.    A solution of the Euler-Lagrange's equation in the form of $A(t)\q$ is called a  relative  equilibrium if  $A(t)$   is a 1-parameter subgroup  of the group $G$. 
     It is well-known that the corresponding configuration $\q$ is a  critical point of the \emph{augmented potential}
\[ U_\xi(\q) =U(\q) +K(\xi_M(\q)),  \]    
    where $\xi$ belongs to the Lie algebra of $G$ and $\xi_M(\q)=\frac{d}{ds}|_{s=0} \exp(s\xi ) \q$ is 
    the vector field on $M$ generated by $\xi$ \cite{Mar92, Sma70-2}.

   The curved $n$-body problem in $\S^3$ (resp. $\H^3$) is a simple mechanical system with symmetry 
  $O(4)$ (resp. $O(3,1)$), the set of matrices that keeps the inner product in $\R^4$ (resp. $\R^{3,1}$). Let   $\xi$ be  some element in the Lie algebra of $O(4)$ ($O(3,1)$). Then  
  the 1-parameter subgroup of $O(4)$ ($O(3,1)$) takes the form of $\exp(t\xi)$, and the corresponding vector field  on the configuration space is $(\xi \q_1, ..., \xi \q_n)$.  The  augmented potential takes the form of  
    \[ U_{1} +   \frac{1}{2} \sum_{i=1}^nm_i \xi \q_i \cdot \xi \q_i, \ \   {\rm or} \  \  \  U_{-1} +\frac{1}{2} \sum_{i=1}^n m_i \xi \q_i \cdot \xi \q_i. \]
   This  coordinates-free way is adopted in the study of  relative  equilibria in the   Newtonian $n$-body problem in higher dimensions, see \cite{Che13, Pal80}.  

    It is convenient to use coordinates for our problem. Note that each 1-parameter subgroup is  conjugate to 
    \begin{equation}\label{equ:1-pg}
 \begin{split}
 & A_{\alpha, \beta}(t)= \begin{bmatrix}
 \cos\alpha t & -\sin \alpha t&0 &0\\
 \sin \alpha t&\cos\alpha t&0&0\\
 0&0&\cos\beta t & -\sin \beta t\\
 0&0&\sin \beta t&\cos\beta t
 \end{bmatrix} {\rm in } \   O(4) \ {\rm  and }, \ \ \\
 	&B_{\alpha, \beta}(t)= \begin{bmatrix}
 	\cos\alpha t & -\sin \alpha t&0 &0\\
 	\sin \alpha t&\cos\alpha t&0&0\\
 	0&0&\cosh \beta t & \sinh \beta t\\
 	0&0&\sinh \beta t&\cosh \beta t
 	\end{bmatrix} {\rm in } \   O(3,1).
 	\end{split}
 	\end{equation}
   We have neglected the  1-parameter subgroups of $SO(3,1)$ that represent the  parabolic rotations since they do not lead  to  relative  equilibria of the curved n-body problem, \cite{Dia13-1, DPV12}. Following Diacu,  \cite{Dia13-1},  in  $\S^3$,  we call relative  equilibria   \emph{elliptic} if only one of  $\a$ and $\beta$ is nonzero, and \emph{elliptic-elliptic} if $\a\b\ne0$;  in  $\H^3$, we  call relative  equilibria  \emph{elliptic} if $\a\ne0, \beta=0$, \emph{hyperbolic} if $\a=0, \beta\ne0$,  and  \emph{elliptic-hyperbolic} if $\a\b\ne0$.

    The  Lie algebra elements  corresponding to  $A_{\alpha, \beta}(t)$ and  $B_{\alpha, \beta}(t)$ are 
    \[  \begin{bmatrix}0&-\alpha&0&0\\
    \alpha&0&0&0\\ 0&0&0&-\beta\\0&0&\beta&0
    \end{bmatrix},    \   \begin{bmatrix}0&-\alpha&0&0\\
    \alpha&0&0&0\\ 0&0&0&\beta\\0&0&\beta&0
    \end{bmatrix}  
     \] 
   respectively.  Hence, the function  $K(\xi \q_1, ..., \xi \q_n)$ for the $\S^3$ case is $ \alpha^2 \sum_{i=1}^n m_i (x_i^2+y_i^2)/2+  \beta^2 \sum_{i=1}^n m_i (z_i^2+w_i^2)/2$. Note that $\q_i \cdot \q_i=1$. The function reduces to 
   \[ \frac{\alpha^2-\beta^2}{2}\sum_{i=1}^n m_i (x_i^2+y_i^2)+\frac{\beta^2}{2} \sum_{i=1}^n m_i.  \]
  Similarly, the function   
    $K(\xi \q_1, ..., \xi \q_n)$ for the $\H^3$ case reduces to 
   \[  \frac{\alpha^2+\beta^2}{2}\sum_{i=1}^n m_i (x_i^2+y_i^2)+\frac{\beta^2}{2} \sum_{i=1}^n m_i.  \] 
    
  Let $I_{1}(\q)= \sum _{i=1}^n m_i (x_i^2+y_i^2)$ (resp. $I_{-1}(\q)= \sum _{i=1}^n m_i (x_i^2+y_i^2)$) be the \emph{momentum  of inertia}  for configurations $\q$ in $\S^3$ (resp. $\H^3$). Let 
  \[ S_c^+= I_1^{-1}(c)-\D_+,  \ \ S_c^-= I_{-1}^{-1}(c)-\D_-. \]  
 We have the following criterion of  relative  equilibria. 

 \begin{theorem}[\cite{DSZ18}]\label{thm:recc}
	For the curved $n$-body problem, 
	 $A_{\alpha, \beta}(t)\q$ ($B_{\alpha, \beta}(t)\q$) is a  relative  equilibrium if and only if the configuration $\q$ is a critical point of 
	\[ U_1(\q)+\frac{\alpha^2-\beta^2}{2} \sum _{i=1}^n m_i (x_i^2+y_i^2),  \ {\rm or} \   \ U_{-1}(\q)+\frac{\alpha^2+\beta^2}{2} \sum _{i=1}^n m_i (x_i^2+y_i^2).  \]
\end{theorem}

Thus all relative  equilibria, no matter it is elliptic, elliptic-elliptic, hyperbolic, or elliptic-hyperbolic, can be obtained by finding configurations that are critical points of one function 
\begin{equation}\label{equ:potential}
U_1-\lambda I_1, \ {\rm or} \  U_1|_{S_c^+}, \ ({\rm resp. } \ U_{-1}-\lambda I_{-1}, \ {\rm or} \  U_{-1}|_{S_c^-}). 
\end{equation}
This unity is obtained   by restricting the  relative equilibria to those in the form of $A_{\alpha, \beta}(t)\q$ ($B_{\alpha, \beta}(t)\q$) and by studying   relative  equilibria in the 3-dimensional  space.  
For example, in the hyperbolic case, if we study the relative  equilibria  on a 2-dimensional physical space,  then the augmented potentials for the elliptic relative  equilibria  and the hyperbolic ones are different, see \cite{DPV12, GMPR16}. 

Let $\q$  be a critical point of \eqref{equ:potential}. Then it  leads to  infinitely many   relative equilibria $A_{\alpha, \beta}(t)\q$ with $ \lambda$  equals $-\frac{\a^2-\beta^2}{2}$ (resp. $B_{\alpha, \beta}(t)\q$ with $ \lambda =-\frac{\alpha^2+\beta^2}{2}$), which maybe elliptic or elliptic-elliptic (resp. elliptic, hyperbolic, or elliptic-hyperbolic). So to study the relative  equilibria, it is convenient to start with the associated configurations. 
We will  discuss  the relationship between  the critical points of \eqref{equ:potential}   and  relative  equilibria  in the Appendix. 

In $\S^3$, the function $U_1$  has critical points. Besides relative equilibria, such configurations also lead to equilibria  of the system $\q(t)=\q$.  Let $\q$  be such a critical point.  Then it is a critical point of $U_1-\lambda I_1$ with $\lambda=0$, or with $\lambda$ being any real value  if it is  one  configuration that  lies on $\S^1_{xy}\cup S^1_{zw}$, two circles to be defined in Section 2.5  \cite{Dia13-1, DSZ18, YZ19}.  	 The property of these configurations and the other   critical points of \eqref{equ:potential} are  different in many ways. This motivates the following definition. 
		
		\begin{definition}[\cite{DSZ18}]
			A configuration $\q$ is called a \emph{central configuration} if it is a critical point of \eqref{equ:potential}. If it is a critical point of $U_1$, it is a \emph{ special central configuration} or an \emph{equilibrium configuration}; otherwise,  it is an \emph{ordinary central configuration}. The value $\lambda$ in \eqref{equ:potential} is the \emph{multiplier}. 
		\end{definition}

 The central configurations defined above are deduced from the relative equilibria, and may not have some nice properties possessed  by the well-known central configurations  of the   Newtonian $n$-body problem. One can refer to \cite{DSZ18} for more of their properties. 
  We only  discuss the ordinary  central  configurations in the main part of this paper.


\subsection{The counting of  ordinary central configurations}  Thanks to Theorem \ref{thm:recc},  we can discuss the counting of ordinary central configurations. 
Note that the set of ordinary central configurations in $\S^3$  (resp. $\H^3$) is invariant under the action of $SO(2)\times SO(2)$  (resp. $SO(2)\times SO^+(1, 1)$) on the configuration space. The group  $SO^+(1,1)$ is  $\{\begin{bmatrix}
\cosh s  & \sinh  s \\
\sinh s &\cosh s 
\end{bmatrix}|  s\in \R \}$, the identity component of $SO(1,1) $.   The symmetry acts on the  configuration space in the following way.   For instance, in $\H^3$, 
let  $\chi=(\chi_1, \chi_2)\in SO(2)\times SO^+(1,1) $. Then 
\[ \chi \q = (\chi \q_1, ..., \chi \q_n), \  \chi \q_i=(\chi_1 (x_i,y_i), \chi_2 (z_i,w_i) ).   \]
The quotient of the set of ordinary central configurations under $SO(2)\times SO(2)$ (resp. $SO(2)\times SO^+(1,1)$) will be called the set of \emph{classes of  ordinary central configurations}.      

The major difference between the ordinary central configurations and the configurations of relative  equilibrium in the Newtonian case is the lack of homothety symmetry. For the   Newtonian $n$-body problem, by the homothety symmetry, the set of configurations of relative  equilibrium on $I_0^{-1}(c)$  equivalent to that on  $I_0^{-1}(1)$. So
it is enough to do the   counting just on $I_0^{-1}(1)$.  For the curved $n$-body problem, the structure of  the set of  ordinary central configurations 
   depends on the value of $I_{\pm 1}(\q)$ in an essential way.    For instance, for two given masses in
   $\S^3$,  the number of  ordinary central configurations    varies as the  value of $I_1$ varies \cite{DSZ18}.  It is also easy to see the existence of critical points of $U_1|_{S^+_c}$ and $U_{-1}|_{S^-_c}$ for each $c$ in some interval. 
   \begin{corollary}[\cite{DSZ18}]
   	For any given $n\ge 2$  masses, there are infinitely many classes of  ordinary central configurations in the curved $n$-body problem in $\S^3$ and $\H^3$. 
   \end{corollary}
   
   Thus, to make the  counting  problem reasonable, we propose to count  ordinary central configurations on  $S^+_c$ ($S^-_c$) for different value of $c$.  To imitate the finiteness problem on configurations of relative equilibrium of 
   the Newtonian $n$-body problem \cite{Sma70-3}, we ask: Are there always only finitely many  ordinary central configurations classes on  $S^+_c$ ($S^-_c$) for  the curved $n$-body problem for almost all choices of masses $(m_1, ..., m_n)$? The answer is negative for some  choices of masses. For example,   for  two masses $m_1=m_2$
   in $\S^3$, there are infinitely many classes 
   of  ordinary central configurations  on $S^+_{m_1}$, see Section 10 of  \cite{DSZ18}. 
   

\subsection{Main results}
We are interested in the  investigation of  the set of the  ordinary central configurations. We first consider  the compactness, then   focus on the counting   of  geodesic  ordinary central configurations and their Morse index.  We postpone the proofs of  Proposition \ref{pro:h2cccomp_lam}, Theorem \ref{thm:h3cccomp} and \ref{thm:s3cccomp} to Section \ref{sec:com}, the proofs of Theorem\ref{thm:cc_h1&ind} and \ref{thm:cc_s1&ind} to Section  \ref{sec:ind}. 

In $\H^3$,   similar to the  singular set   of the Newtonian n-body problem, a point in $\D_-$ can be written as
$$X=(\q'_1, ..., \q'_{k_1}, \q'_{k_1+1}, ..., \q'_{k_2}, ..., \q'_{k_{s-1}+1}, ..., \q'_{k_s}), $$ 
$\q'_k =(x'_k, y'_k, z'_k, w'_k)$, where we have grouped the equal terms $\q'_1=...= \q'_{k_1}, \q'_{k_1+1}=...= \q'_{k_2}, ..., \q'_{k_{s-1}+1}= ...= \q'_{k_s},  k_s=n$.  We call each group of particles a \emph{cluster} of $X$. Denote by $\Lambda_i$ the index set of the $i$-th cluster, i.e., $\Lambda_i=\{k_{i-1}+1, ..., k_i \}$,  $1\le i\le  s$. Let  $|\Lambda_1|=k_1, ..., |\Lambda_i|=k_i -k_{i-1},  ..., |\Lambda_s|=n-k_{s-1}$.     Let $\q(l)=(\q_{1(l)}, ..., \q_{n(l)}), l=1, ..., \infty, $ be   a sequence of  ordinary central configurations that converges to $X$.  Let $\q_{i(l)} = (x_{i(l)}, y_{i(l)}, z_{i(l)}, w_{i(l)})$. Denote by  $\lambda(l)$ the multiplier of $\q(l)$.  

   \begin{proposition}\label{pro:h2cccomp_lam}
	Given $n$ masses in $\H^3$, 
	if there is a sequence of    ordinary central configurations  that converges to  some point $X\in \Delta_-$, 
	then the sequence of  multipliers approaches $-\infty$.  
\end{proposition}

 The set of ordinary central configurations is not compact. For example, consider 
 the following  regular polygonal  configuration formed by n equal masses
 \[  \q_i=(\sinh \th \cos \frac{i 2 \pi}{n}, \sinh \th \cos \frac{i 2 \pi}{n},  0, \cosh \th), \  i=1, ..., n. \]
 It is easy to check that it is an ordinary central configuration for any $\th \in (0, \infty)$. These ordinary central configurations are not in one class since there is no homothety symmetry in the set of ordinary central configurations. 
 As $\th\to 0$, the configuration converges to a singular configuration.  Note that the momentum of inertia of that singular point is $0$.

\begin{theorem}\label{thm:h3cccomp}
	Given $n$ masses in $\H^3$ and any point $X\in \D_-$ with $I_{-1}(X)=c>0$, there is a neighborhood of $X$ in which there is no  ordinary central configuration. 
\end{theorem}

    \begin{remark}\label{rem:tib}
	Consider the subset of ordinary central configurations on $\H^2_{xyw}$, the intersection of $\H^3$ and the hyperplane $z=0$,  with the property that all particles  lie  on a  same plane perpendicular to the $w$-axis and the value of the multiplier is fixed. 
	Tibboel   \cite{Tib14}  proved that this subset is compact in the configuration space. Our result is stronger. 
\end{remark}

In $\S^3$, a point in $\D_+$ can be written as
\begin{equation}\label{equ:sings3}
X=(\q'_1, ..., \q'_{k_1}, \q'_{k_1+1}, ..., \q'_{k_2}, ..., \q'_{k_{2s-1}+1}, ..., \q'_{k_{2s}}), 
\end{equation}
where we have grouped the equal and antipodal  terms $\q'_1=...= \q'_{k_1}, \q'_{k_1+1}=...= \q'_{k_2}, -\q'_{k1}=\q'_{k_2},  ..., \q'_{k_{2s-1}+1}= ...= \q'_{k_{2s}},  -\q_{k_{2s-1}}=\q'_{k_{2s}},  k_{2s}=n$.  
 If $|\Lambda_{k}|>1$, particles in the $k$-th cluster form a \emph{collision singular configuration}. If $|\Lambda_{2i-1}|=|\Lambda_{2i}|=1$, particles in the two cluster form an  \emph{antipodal singular configuration}. If  
$|\Lambda_{2i-1}|\ge 2$  and $|\Lambda_{2i}|\ge 1$,  particles in the two cluster form a  \emph{collision-antipodal singular configuration}. 

As in the case of $\H^3$, the set of ordinary central configurations 
is  not compact. For instance, consider the regular polygonal configurations formed by $n$ equal masses at position 
\[  \q_i=(\sin \th \cos \frac{i 2 \pi}{n}, \sin \th \cos \frac{i 2 \pi}{n},  \cos \th, 0), \  i=1, ..., n. \]
They approach a collision singular configuration with momentum of inertia $0$ as $\th \to 0$. The situation is more complicated than that in $\H^3$.  By the transform $\tau\in O(4)$ of Theorem \ref{thm:s2cc},  we get  a  1-parameter family  of ordinary central configurations approaching a singular point with momentum of inertia $n$. Another examples, consider three masses $m_1=m, m_2=m_3=M$ at position
\[\q_1=(1, 0,0,0), \ \q_2=( -\cos \th, 0, \sin \th, 0), \ \q_3=( -\cos \th, 0, -\sin \th, 0).  \]
It is an ordinary central configuration for any $\th\ne0$ by equation \eqref{equ:ree-yzws3}.    As $\th \to 0$, the configurations approaches a collision-antipodal singular configuration. Let $\S_{xy}^1:=\{(x,y,z,w) \in \S^3: z=w=0\}$, $\S_{zw}^1:=\{(x,y,z,w) \in \S^3: x=y=0\}$.   Note that the singular configurations of  the  three examples all lie on the union of two circles,  
$\S_{xy}^1\cup\S_{zw}^1$.

We consider only a subset, denoted by $\mathcal A$,  of  $\D_+$ with the following  two properties:  1. if $X\in \mathcal A$, then not all particles of $X$ lie on $\S^1_{xy}\cup \S^1_{zw}$; 2.  $X$ contains collision singular sub configuration or antipodal  singular sub configuration, that is, there is some $i$ such that $|\Lambda_{2i-1}|\ge 2, |\Lambda_{2i}|=0$,  or, $|\Lambda_{2i-1}|=|\Lambda_{2i}|= 1$.  

\begin{theorem}\label{thm:s3cccomp}
	Given $n$  masses in $\S^3$ and any point $X\in \mathcal A$, there is a neighborhood of $X$ in which there is no    ordinary central configuration. 
\end{theorem}

   We now consider the geodesic configurations. 
   Let us introduce some \textbf{notations}.  A \textit{geodesic   central configuration} is one  for which all particles 
   lie  on a same  geodesic. A \textit{2-dimensional   central configuration} is one  for which all particles 
   lie  on a same  2-dimensional  great sphere  but not on a same geodesic. Denote by $\S_{xyz}^2$ (resp. $\S_{xzw}^2$) the 2-dimensional  great sphere
   intersected by $\S^3$ and the hyperplane $w=0$ (resp. $y=0$). Denote by   $\H_{xyw}^2$ the intersection of $\H^3$ and  the hyperplane $z=0$. Let  $\S_{xz}^1:=\{(x,y,z,w) \in \S^3: y=w=0\}$, $\H_{xw}^1:=\{(x,y,z,w) \in \H^3: y=z=0\}$. We have  some  related preliminary results.  
   
     \begin{theorem}[\cite{DSZ18}]\label{thm:mc}
   	Let $\q=(\q_1,\dots,\q_n),\ \q_i=(x_i,y_i,z_i,w_i),\ i=1,...,n,$ be an ordinary central configuration in $\H^3$  $(\S^3)$.
   	Then we have the relationships
   	\begin{equation*}
   	\sum_{i=1}^n m_ix_iz_i=\sum_{i=1}^n m_ix_iw_i=\sum_{i=1}^n m_iy_iz_i=\sum_{i=1}^n m_iy_iw_i=0.
   	\end{equation*}
   \end{theorem}

   \begin{remark}\label{rem:mc}
   	The above relationships have  been found  in \cite{GMPR16,Kil98} for two-body ordinary central configurations, where it reads as $m_1\sin 2\th_1=m_2\sin 2\th_2$ or $m_1\sinh 2\th_1=m_2\sinh 2\th_2$. 
   	Recall  that  configurations of relative equilibrium in $\R^2$ have center of mass at the origin, i.e., $\sum_{i=1}^n m_ix_i=\sum_{i=1}^n m_iy_i=0.$ Theorem \ref{thm:mc} can be viewed as an analogy of that fact. 
   \end{remark}

  	\begin{theorem}[\cite{DSZ18}] \label{thm:s2cc}
  	In $\S^3$, each geodesic (resp. 2-dimensional ) ordinary central configuration 
  	 is equivalent to  one  on $\S^1_{xz}$ (resp. $\S_{xyz}^2$ or $\S_{xzw}^2$). 
  	Any  ordinary central configuration  with multiplier $\lambda$  is  mapped to one  ordinary central configuration with multiplier $-\lambda$  by $\tau\in O(4)$, where $\tau (x, y, z,w)=(z, w, x, y)$. 
  \end{theorem}

    \begin{theorem}[\cite{DSZ18, ZZ16}] \label{thm:h2cc}
    	In $\H^3$, each  ordinary central configuration is  equivalent to  one on $\H^2_{xyw}$.  Each geodesic  ordinary central configuration  is equivalent to  one on $\H^1_{xw}$.
    \end{theorem}

 Thus, for geodesic (resp. 2-dimensional ) ordinary central configurations,  it is enough to  study the ones on $\S_{xz}^1$  and $\H_{xw}^1$ (resp.  $\S_{xyz}^2$  and $\H_{xyw}^2$). 
  We will call the those special sub manifolds $\S^1$  and $\H^1$    
 (resp.   $\S^2$ and $\H^2$). 
    
 The   ordinary central configurations  on $\S^2$  (resp. $\H^2$) are critical points of $U_1(\q)$ (resp. $U_{-1}(\q)$) restricted  on $S^+_c\cap(\S^2)^n$ (resp. $S^-_c\cap(\S^2)^n$). With a slight abuse of notation, we still call them $S^+_c$ and   $S^-_c$.
   Further more,  the classes of  ordinary central configurations  in  $I^{-1}_1(c)$ (resp. $I^{-1}_{-1}(c)$) 
    correspond in $1$-$1$  manner to the critical points of   $U_1(\q)$ (resp. $U_{-1}(\q)$) restricted on $S^+_c/S^1$ (resp. $S^-_c/S^1$).  Recall that for the   Newtonian $n$-body problem, the classes  of  configurations of relative equilibrium correspond in $1$-$1$ manner to the critical points of   $U_0(\q)$  restricted on $\mathcal{S}/S^1$, where 
    \[  \mathcal{S} =   \{ \q\in(\R^2)^n-\Delta\ \! | \ \! I_0(\q)=1\}.   \]

    For the $\H^2$ case, the set $S^-_c/S^1$ $(c>0)$ is obviously diffeomorphic  to $\mathcal{S}/S^1$.  Thus, $S^-_c/S^1$ is smooth $(2n-2)$-dimensional manifold. 
 
    \begin{theorem} \label{thm:cc_h1&ind}
    	Given $n$ masses  on  $\H^1$ and any positive value of $c$,  there are exactly $n!/2$ 
    	geodesic  ordinary central configurations  in $S^-_c$, one for each ordering of  the masses along $\H^1$.  At each of them, the  Hessian of $U_{-1}|_{S^-_c/S^1}$  has inertia  $ (n_0, n_+, n_-)= (0, n, n-2)$.
    \end{theorem}
    

  For the $\S^2$ case, the set $S_c^+/S^1$ is more complicated. Assume that $m_1$ is the smallest mass. 
  
  \begin{proposition}
 For given $n$ masses, the critical values of the function   $I_1(\q)$ are $\{\sum \epsilon_i m_i | \epsilon_i =0 \ {\rm or } \  1\}$. Assume that $m_1$ is the smallest mass. 
  	If  $c< m_1$, then $S^+_c$ is a smooth $(2n-2)$-dimensional manifold  with  $2^n$ components.  
  \end{proposition}
\begin{proof}
Let $f(x,y,z)=x^2+y^2,  \  (x,y,z)\in \S^2$. Obviously, the critical points of $f$  consist of   the equator and the  two poles. 
	Thus, the critical points of $I_1(\q)$ are 
	\[  \{  \q\in (\S^2)^n|  \q_i =(0, 0, \pm 1) \ {\rm or } 
	\ (\cos \vp, \sin \vp, 0), \  i=1, ..., n.  \},    \]
	 which  gives the set of  critical values. If c is less than the first critical value $m_1$, no particle of configurations in $S^+_c$ can lie on the equator. Then each $z_i$ must  be either positive or negative, which implies that there are $2^n$ components of $S^+_c$.   
	\end{proof}

We will only consider  geodesic ordinary central configurations 
on the following   component of $S^+_c$,  
\[  \mathcal{M}_c =\{  \q\in S^+_c|   z_i >0, 1=1, ..., n \},  \ 0<c< m_1.\]  
Note that $\mathcal{M}_c/S^1$  is  diffeomorphic  to $\mathcal{S}/S^1$.  The multiplier of ordinary central configurations on $\mathcal M_c$ must be negative, see  Proposition \ref{pro:s2lam}.
 
 \begin{theorem} \label{thm:cc_s1&ind}
 	Given $n$ masses  on  $\S^1$ and any value of $c\in (0, \frac{m_1}{2})$,  there are exactly $n!/2$ 
 	geodesic  ordinary central configurations  on $\mathcal{M}_c$, one for each ordering of  the masses.  Provided that $0<c<\frac{m_1}{4}$,  at each of them, the  Hessian of $U_{1}|_{\mathcal M_c/S^1}$ has inertia  $ (n_0, n_+, n_-)= (0, n, n-2)$.
 \end{theorem}

The  restriction of the value of $c$ might not be sharp, but it is necessary.  In the two-body case,  if $m_1<m_2$, then there is no ordinary central configuration on $\mathcal M_c$ for $c\in[m_1, m_2]$.  If  $m_1=m_2=m$,  there is a continuum  of  ordinary central configurations on $\mathcal M_m$, and  each of them has  inertias $ (n_0, n_+, n_-)= (1, 1, 0)$, 
see Section 10 of  \cite{DSZ18}.

Recall that the  inertia  of each collinear configuration  of relative equilibrium of the Newtonian n-body problem is also  $ (n_0, n_+, n_-)= (0, n, n-2)$
if we study the Hessian of $U_0$ on $\{ \q: I_0(\q)=1  \}/S^1$ instead of $\{ \q: I_0(\q)=1, \sum m_i \q_i=0  \}/S^1$, \cite{Moe-N}.   So Theorem \ref{thm:cc_h1&ind} and \ref{thm:cc_s1&ind} confirm the following general belief, namely,  many  properties of configurations of relative equilibrium of the Newtonian n-body problem also  hold for ordinary central configurations  of the curved n-body problem provided that the value of $I_{\pm1}(\q)$ is small enough.  In this particular case, i.e., the number and inertia of geodesic ordinary central configurations,  
Theorem \ref{thm:cc_h1&ind} indicates that the restriction of  small   value of $I_{-1}(\q)$ is not necessary  on $\H^2$;  
Theorem  \ref{thm:cc_s1&ind} gives an upper bound of the value of $I_{1}(\q)$ such that the corresponding results hold on $\S^2$.


The three manifolds $\mathcal{S}/S^1$, $S^-_c/S^1$ and $\mathcal{M}_c/S^1$ ($0<c<m_1$) are diffeomorphic, so 
 share the same Poincar\'{e} polynomial.  The number of  geodesic  ordinary central configurations  on $\mathcal{M}_c/S^1$ ($S^-_c/S^1$)  and Morse index of them  are the same as that on $\mathcal{S}/S^1$ of the   Newtonian $n$-body problem. By Theorem \ref{thm:h3cccomp} and \ref{thm:s3cccomp},  the set of ordinary central configurations on $\mathcal{M}_c/S^1$ ($S^-_c/S^1$) are compact.  Thus, we can apply  the argument  of Palmore \cite{Pal73, Sma70-3} to  obtain  the following estimation  on the number of critical points of $U_1(\q)$ (resp. $U_{-1}(\q)$)  restricted on $\mathcal M_c$ (resp. $S^-_c$).

\begin{corollary}
	Suppose that for a certain choice of masses of the curved $n$-body problem on  $\S^2$ (resp. $\H^2$) all   ordinary central configurations   are non-degenerate critical points of $U_1|_{\mathcal{M}_c/S^1}$ (resp. $U_{-1}|_{S^-_c/S^1}$). 
	 Then  in  $\mathcal{M}_c$ $(0<c<\frac{m_1}{4})$  (resp. $S^-_c$ $(c>0)$),   there are at least  
$\frac{(3n-4)(n-1)!}{2}$
 ordinary central configurations,  of which at least 
$ \frac{(2n-4)(n-1)!}{2}$
are non-geodesic. 
\end{corollary}

     \section{proof of  Theorem \ref{thm:h3cccomp} and Theorem \ref{thm:s3cccomp}  } \label{sec:com}
   
    Theorem \ref{thm:h3cccomp} and  \ref{thm:s3cccomp}   are  analogous  to   Shub's  lemma
in the   Newtonian $n$-body problem. It is first proved by Shub \cite{Shu70}. Moeckel gives a shorter proof in \cite{Moe-N}.  The idea of Moeckel  is applicable  if we   use the following  Cartesian coordinate system for $(\S^3)^n$ ($(\H^3)^n$).

Recall that we have  written the Euler-Lagrange equation in Section \ref{sec:cnbpeuler}   with the Cartesian coordinates. For each particle, we use the four coordinates  $(x_i, y_i, z_i, w_i)$ to represent its position.  That coordinate system is redundant. 
 Three coordinates are enough to  represent the positions of each particle. For instance, if $x_i\ne0$, then $\q_i$ is in an open region of $\S^3$ for that $(y,z,w)$ can serve as a local chart. Then $\frac{ \partial (U_1-\lambda I_1)}{\partial \q_i}=0$ is equivalent to $\frac{ \partial (U_1-\lambda I_1)}{\partial y_i}=\frac{ \partial (U_1-\lambda I_1)}{\partial z_i}=\frac{ \partial (U_1-\lambda I_1)}{\partial w_i}=0$. 
Since $x_i^2+y_i^2=1-z_i^2-w_i^2$, so we have 
\[  \frac{\partial I}{\partial y_i} = 0, \  \frac{\partial I}{\partial z_i} = -2m_iz_i, \   \frac{\partial I}{\partial w_i} = -2m_iw_i. \]
By $\cos d_{ij}=x_ix_j+y_iy_j+z_iz_j+w_iw_j$, we obtain 
\[   \frac{\partial \cot d_{ij}}{\partial y_i} =  \frac{-1}{\sin^2 d_{ij}}  \frac{\partial d_{ij}}{\partial y_i} =   \frac{1}{\sin^3 d_{ij}} ( \frac{\partial x_i}{\partial y_i} x_j + y_j)  =   \frac{1}{\sin^3 d_{ij}} (  y_j -\frac{x_j}{ x_i} y_i).  \]
Similarly, we have $ \frac{\partial \cot d_{ij}}{\partial z_i} = \frac{1}{\sin^3 d_{ij}} (  z_j -\frac{x_j}{ x_i} z_i),  \frac{\partial \cot d_{ij}}{\partial w_i} = \frac{1}{\sin^3 d_{ij}} (  w_j -\frac{x_j}{ x_i} w_i)$. 
Thus, 
 the  central configuration equations for $\q_i$ can be written as 
\begin{equation} \label{equ:ree-yzws3}
\sum_{j\ne i} m_im_j\frac{\v_j -\frac{x_j}{x_i}\v_i}{\sin^3 d_{ij}} =-\lambda 2m_i(0, z_i, w_i), \ {\rm \ if }\  x_i\ne 0, 
\end{equation}  
where $\v_i=(y_i, z_i, w_i)$. Similarly, if $w_i\ne0$, the central configuration equation for $\q_i$ can be written as 
\begin{equation} \label{equ:ree-xyzs3}
\sum_{j\ne i} m_im_j\frac{\u_j -\frac{w_j}{w_i}\u_i}{\sin^3 d_{ij}} =\lambda 2m_i(x_i, y_i, 0), \ {\rm \ if }\  w_i\ne 0, 
\end{equation}  
where $\u_i= (x_i, y_i, z_i)$. Similarly, the equations can be written in   other forms if $y_i\ne 0$ or $z_i\ne 0$.  In $\H^3$,  $(x, y, z)$ serve as a global chart.  Thus, the central configuration equations  in $\H^3$  can be written as 
\begin{equation} \label{equ:ree-xyzh3}
\sum_{j\ne i} m_im_j\frac{\u_j -\frac{w_j}{w_i}\u_i}{\sinh^3 d_{ij}} =\lambda 2m_i(x_i, y_i, 0), \ i=1, ..., n. 
\end{equation}

   \begin{proof} [Proof of Proposition \ref{pro:h2cccomp_lam}]
   View the two sides of the central configuration equations \eqref{equ:ree-xyzh3} as vectors in $\R^{3}$, and multiply on both sides by $\u_i$. We obtain 
   	\begin{align*}
   	2\lambda m_i (x_i^2 +y_i^2)= \sum _{j\ne i}\frac{\u_j\cdot \u_i -\frac{w_j}{w_i}(w_i^2 -1)}{\sinh^3 d_{ij}} =\sum _{j\ne i}\frac{ \frac{w_j}{w_i} - \cosh d_{ij}}{\sinh^3 d_{ij}}. 
   	\end{align*}
   	Assume that $|\Lambda_1|=k_1\ge 2$.  Denote by $\q(l)$  
   	the sequence of  ordinary central configurations  that converges to $X$.  For each $l$,  assume that $\q_{i(l)}$ of the first cluster has 
   	 the biggest value of $w$, i.e., $w_{i(l)}\ge w_{j(l)}$ for any $1\le j(l)\le k_1$. The above equality for  $\q_{i(l)}$ is
   	\[  2\lambda(l) m_{i(l)} (x_{i(l)}^2 +y_{i(l)}^2)
   	=\sum _{j(l)\ne i(l), j(l)\in \Lambda_1}\frac{ \frac{w_{j(l)}}{w_{i(l)}} - \cosh d_{ij(l)}}{\sinh^3 d_{ij(l)}}+O(1).\]
   	As $l\to \infty$, each term in  the above sum  approaches $-\infty$ since
   	$$\frac{ \frac{w_{j(l)}}{w_{i(l)}} - \cosh d_{ij(l)}}{\sinh^3 d_{ij(l)}} \le \frac{ 1- \cosh d_{ij(l)}}{\sinh^3 d_{ij(l)}}=\frac{ -\frac{d^2_{ij(l)}}{2} +O(d^4_{ij(l)})}{\sinh^3 d_{ij(l)} }\to -\infty. $$
   	The value of $ x_{i(l)}^2 +y_{i(l)}^2$ is obviously  bounded above.  Hence,  the multiplier of the sequence of  ordinary central configurations   approaches  $-\infty$. 
   \end{proof}

   \begin{proof}[Proof of Theorem \ref{thm:h3cccomp}]
   	Since $I_{-1}(X)=c>0$,  not all clusters of $X$ are at $(0,0, 0, 1)$. Assume that the first cluster is not at $(0,0, 0, 1)$. Note that  the central configuration equations  can be written as 
   	$\sum _{j\ne i}\frac{w_i\u_j -w_j\u_i}{\sinh^3 d_{ij}}=2\lambda m_iw_i (x_i, y_i, 0), i=1, ..., n.$  Then adding the equations corresponding to particles in the first cluster, we obtain 
   	\[\sum_{i\in \Lambda_1} \sum _{j\ne i}\frac{w_i\u_j -w_j\u_i}{\sinh^3 d_{ij}}=  \sum_{i\in \Lambda_1} \sum _{j\notin \Lambda_1 }\frac{w_i\u_j -w_j\u_i}{\sinh^3 d_{ij}}=2\lambda  \sum_{i\in \Lambda_1}  m_iw_i (x_i, y_i, 0).  \]
   	Assume that there is  a sequence of  ordinary central configurations $\q(l)$  that converges to $X$. The above equality reads  $O(1)=\infty$. This contradiction  shows that  there is a neighborhood of $X$ in which there is no    ordinary central configuration. 
   \end{proof}

\begin{proof}[Proof of Theorem \ref{thm:s3cccomp}]
	
Let  $X\in \mathcal A$.  Assume that the  first cluster of $X$ does not lie on $\S^1_{xy}\cup \S^1_{zw}$, i.e., $\q_1'\notin\S^1_{xy}\cup \S^1_{zw}$.  Since $z_1'^2 +w_1'^2\ne 0$, we can assume that $w_1'\ne 0$.  
 Let $\q$ be an ordinary central configuration 
 close to $X$. Then
  the central configuration equations for particles in the first two  clusters  can be written as 
   	$$\sum _{j\ne i,  j\in\Lambda_1\cup \Lambda_2}\frac{w_i\u_j -w_j\u_i}{\sin^3 d_{ij}}+O(1)=2\lambda m_iw_i (x_i, y_i,  0), i\in \Lambda_1\cup \Lambda_2, $$
   	where the $O(1)$ term corresponds to  interactions between $\q_i$ and particles of the other $2s-2$ clusters. Adding those equations, 
    we obtain 
   	\[\sum_{i\in\Lambda_1\cup \Lambda_2} \left(\sum _{j\ne i, j\in \Lambda_1\cup \Lambda_2}\frac{w_i\u_j -w_j\u_i}{\sinh^3 d_{ij}}+O(1)\right)=  O(1)=2\lambda \sum_{i\in \Lambda_1\cup \Lambda_2}   m_iw_i (x_i, y_i,  0).   \]
   	  Assume that there is  a sequence of  ordinary central configurations $\q(l)$  that converges to $X$.  Since $X$ contains sub configuration which is collision singular or antipodal singular, we can prove that the absolute value of multiplier $|\lambda(l)|$ goes to infinite by  argument similar to that in the proof of Proposition \ref{pro:h2cccomp_lam}. 
   	Note that  
   	$$\sum_{i\in \Lambda_1\cup \Lambda_2}  m_iw_i (x_i, y_i,  0)\to  \sum_{i=1}^{k_1} m_i(x_1'w_1',y_1'w_1',0 )+\sum_{i=k_1+1}^{k_2} m_i(x_1'w_1',y_1'w_1',0 ) \ne 0$$
   	 since $x_1'^2 +y_1'^2\ne 0$ and $w_1'\ne 0$. The above equality reads  $O(1)=\infty$. This contradiction  shows that  there is a neighborhood of $X$ in which there is no    ordinary central configuration. 
   \end{proof}

       \section{proof of  Theorem \ref{thm:cc_h1&ind} and Theorem \ref{thm:cc_s1&ind}  } \label{sec:ind}
    
     In the   Newtonian $n$-body problem, the number of collinear configurations of relative equilibrium  is first found  to be $\frac{n!}{2}$ by Moulton \cite{Mou10}, then Smale gives a shorter proof \cite{Sma70-2}.  The index of them is 
      $n-2$ \cite{Pac87}. 
 The idea in \cite{Pac87}, due to Conley,   is applicable  if we   use an angle coordinate system for $(\S^2)^n$ ($(\H^2)^n$).
    Theorem \ref{thm:cc_h1&ind} (the  $\H^2$ case) is proved in  the first subsection, divided into two parts. 
     Theorem \ref{thm:cc_s1&ind} (the  $\S^2$ case) is proved by a similar way with minor modification in the last subsection.

    \subsection{Geodesic  ordinary central configurations  on $\H^2$} \label{subsec:h1}
   \begin{proposition}\label{pro:h2lam}
  	Let $\q$ be an ordinary central configuration on  $\H^2$. Then the multiplier is negative. 
  \end{proposition}
  
  \begin{proof}
  	Assume $w_i\le w_n$ for $i=1, ..., n$. Then $w_n>1$.  Recall that the central configuration equation  \eqref{equ:ree-xyzh3} for $\q_n$ can be written as 
  	$\sum _{j\ne n}\frac{\u_j -\frac{w_j}{w_n}\u_n}{\sinh^3 d_{jn}}=2\lambda m_n \u_n$. Here $\u_i=(x_i, y_i)$ since the configuration is on $\H^2$.  Multiplying $\u_n$ on both sides, the right side becomes $2\lambda m_n (w_n^2-1)$, and the left side becomes 
  	\[  \sum _{j\ne n}\frac{\u_j \cdot \u_n -\frac{w_j}{w_n}(w_n^2-1)}{\sinh^3 d_{jn}}  = \sum _{j\ne n}\frac{\frac{w_j}{w_n} - \cosh d_{jn} }{\sinh^3 d_{jn}}<0.   \]
  	This shows that the multiplier  $\lambda$ is negative.   
  \end{proof}

   We use an angle coordinate system for $\H^2$,  $(\th , \vp)$,  $\th, \vp \in \R $.   The relationship between  Cartesian coordinates $(x,y,w)$ and $(\th , \vp)$ is   
  \[  (x,y,w)=(\sinh \th, \cosh \th \sinh \vp,\cosh \th \cosh \vp) .\]
  Then  $\H^1=\H^1_{xw}$ is parameterized by $(\th, 0)$ and  $(\H^2)^n$ is parameterized by $(\th_1, ..., \th_n, $
  $ \vp_1, ..., \vp_n)$.  
  The momentum of inertia is 
  \[  I_{-1}(\q)= \sum_{i=1}^{n} m_i (x_i^2 +y_i^2)=   \sum_{i=1}^{n} m_i (\sinh^2 \th_i +\cosh^2 \th_i \sinh^2 \vp_i). \] 
  In the above angle coordinates,  a configuration $(\th_1, ..., \th_n, \vp_1, ..., \vp_n)$ is an ordinary central configuration  if $$ \frac{\partial U_{-1}}{\partial \th_i}=\lambda \frac{\partial I_{-1}}{\partial \th_i}, \ \frac{\partial U_{-1}}{\partial \vp_i}=\lambda \frac{\partial I_{-1}}{\partial \vp_i},\ i=1, ..., n.   $$ 
  
 Restricted on $\H^1$, we can just use the theta coordinates $(\th_1, ..., \th_n)$ to parametrize the configurations.  Then 
 	\begin{align*}
 & \q_i=(\sinh\theta_i, 0, \cosh\theta_i),\ \ &d_{ij}=|\theta_i-\theta_j|,\\
  &U(\q)=\sum_{1\le i<j\le n}m_im_j\coth d_{ij}\ \ &
 I(\q)=\sum_{i=1}^nm_i\sinh^2\theta_i.
 	\end{align*}
 A configuration $(\th_1, ..., \th_n)$ is an ordinary central configuration  if $ \frac{\partial U_{-1}}{\partial \th_i}=\lambda \frac{\partial I_{-1}}{\partial \th_i}$, or explicitly, 
  \begin{equation}\label{equ:reeh1}
  \sum_{j\ne i}\frac{m_im_j \sinh(\th_j-\th_i)}{\sinh ^3 d_{ij}}=\lambda m_i \sinh 2\th_i,  \ i=1, ..., n. 
  \end{equation}
\begin{proof}[Proof of Theorem \ref{thm:cc_h1&ind}, Part I]
     The numbers of geodesic  ordinary central configurations  is proved to be $\frac{n!}{2}$ in \cite{DSZ18}.  The argument  is similar to  that of  the    Newtonian $n$-body problem \cite{Sma70-2}.  We briefly repeat the idea  here. 
     
     Restrict the function $U_{-1}$ on the set   $S^-_c\cap (\H^1)^n=\{ \q\in (\H^1)^n-\D_-| I_{-1}(\q) =c \} $, which  has  $n!$ components.   Each component is   homeomorphic to an open ball.  
     On each component, there is at least one minimum. All critical points are minima since the Hessian of $U_{-1}$ at each  of them is  positive definite. Thus, there are $n!$ critical points, and the number of  ordinary central configurations  is $\frac{n!}{2}$ by the $SO(2)$ symmetry. 
     
For completeness, we repeat the proof  of positive definiteness.   Let $\bar  {\q}=(\bar  \th_1, ..., \bar  \th_n)$ be  one critical point of $U_{-1}|_{S^-_c}$. 
     	 Assume that $\bar  \th_1< ...< \bar  \th_n$ and that the multiplier  is $\bar{\lambda}$. Then the Hessian of $U_{-1}|_{S^-_c\cap (\H^1)^n}$ equals 
 $ D^2(U_{-1}-\bar \lambda I_{-1})|_{\bar{\q}} $    
     	 as quadratic  forms on  $T_{\bar {\q}} S^-_c\cap (\H^1)^n$,  
     	 the tangent space of $S^-_c\cap (\H^1)^n$ at $\bar  \q$. 
      By straightforward computations, we obtain 
     	\begin{align*}
     &\ \!D^2 (U_{-1} -\bar \lambda I_{-1})|_{\bar{\q}} \\ 
     	=& 2\begin{bmatrix}
     	\sum\limits_{\substack{j=1, j\ne 1}}^n \frac{m_1m_j\cosh d_{1j}}{\sinh^3d_{1j}} & -\frac{m_1m_2\cosh d_{12}}{\sinh^3d_{12}}& \cdots& -\frac{m_1m_n\cosh d_{1n}}{\sinh^3d_{1n}}\\
     	-\frac{m_2m_1\cosh d_{12}}{\sinh^3d_{12}}&  \sum\limits_{\substack{j=1,j\ne 2}}^n \frac{m_2m_j\cosh d_{2j}}{\sinh^3d_{2j}}&\cdots& -\frac{m_2m_n\cosh d_{2n}}{\sinh^3d_{2n}}\\
     	\cdots&  \cdots& \cdots& \cdots\displaybreak[0]\\
     	-\frac{m_1m_n\cosh d_{1n}}{\sinh^3d_{1n}}&\cdots&\cdots&\sum\limits_{\substack{j=1, j\ne n}}^n \frac{m_nm_j\cosh d_{nj}}{\sinh^3d_{nj}}
     	\end{bmatrix}\\
     	&-2\bar \lambda \begin{bmatrix}
     	m_1 \cosh 2\theta_1&0&\cdots&0\\
     	0&m_2 \cosh 2\theta_2&\cdots&0\\
     	\cdots&  \cdots& \cdots& \cdots\\
     	0&\cdots &\cdots &m_n \cosh 2\theta_n\
     	\end{bmatrix}.
     	\end{align*}
     The second part is   obviously positive definite. 
     	For the  first part,  let us  take any nonzero vector 
     	$\v=(v_1,\cdots, v_n)\in T_{\bar {\q}} S^-_c\cap (\H^1)^n$.  We obtain 
     	\begin{align*}
     	\v^T (D^2U_{-1})\v=\sum_{i=1}^n\sum_{j=1}^n(D^2 U_{-1})_{ij} v_iv_j=\ & 2\sum_{i=1}^{n}\sum\limits_{\substack{j=1\\j\ne i}}^n \frac{m_im_j\cosh d_{ij}}{\sinh^3d_{ij}}v_i^2\\- 2\sum_{i=1}^n\sum \limits_{\substack{j=1\\j\ne i}}^n \frac{m_im_j\cosh d_{ij}}{\sinh^3d_{ij}}v_iv_j
     	=\ &\sum_{i=1}^n\sum\limits_{\substack{j=1\\j\ne i}}^n \frac{m_im_j\cosh d_{ij}}{\sinh^3d_{ij}}(v_i-v_j)^2 \geq 0.
     	\end{align*}
     	Hence, we  conclude that Hessian of $U_{-1}|_{S^-_c\cap (\H^1)^n}$ is positive definite. 
\end{proof}    
We have seen that system \eqref{equ:reeh1} have $n!/2$ solutions.   Each is a minimum  of 
  $U_{-1}|_{S^-_c\cap (\H^1)^n}$.  We are going to study the Hessian  on $S^-_c$ at those critical points.  
  
  Let $\bar  {\q}=(\bar  \th_1, ..., \bar  \th_n, 0, ..., 0)$ be  one of the  $n!/2$ geodesic  ordinary central configurations.  
   Assume that $\bar  \th_1< ...< \bar  \th_n$ and that the multiplier  is $\bar{\lambda}$. Then $\bar  {\q}$ is also a critical point of  $f=U_{-1}-\bar \lambda I_{-1}$.
  Denote by $\mathcal{H}(U_{-1}|_{S^-_c}, \bar  {\q})$  (resp. $\mathcal{H}(f, \bar  {\q})$) the Hessian of $U_{-1}|_{S^-_c}$ (resp. $f$) at the critical point $\bar {\q}$.  
Then we have 
  \[    \mathcal{H}(f, \bar  {\q}) =\mathcal{H}(U_{-1}|_{S^-_c}, \bar  {\q}) \]
  as quadratic  forms on  $T_{\bar {\q}} S^-_c$,   
  the tangent space of $S^-_c$ at $\bar  \q$.

The set $S^-_c$ is  hyper surface in $(\H^2)^n$ with co-dimension 1. 
The tangent space of $(\H^2)^n$ at $\bar  \q$, $T_{\bar {\q}} (\H^2)^n$, is   spanned by  $  \frac{\partial }{\partial \th_1}, ..., \frac{\partial }{\partial \th_n}, \frac{\partial }{\partial \vp_1}, ..., \frac{\partial }{\partial \vp_n}$.   The normal of $S^-_c$  at $\bar {\q}$  
is 
  $ \nabla I_{-1}|_{\bar {\q}} = \sum_{i=1}^{n}m_i \sinh 2\bar \th_i \frac{\partial }{\partial \th_i}, $
   which is in the subspace spanned by $  \frac{\partial }{\partial \th_1}, ..., \frac{\partial }{\partial \th_n}$. Thus, we see $T_{\bar {\q}} S^-_c=V_1 \oplus V_2$,
   \[  
   \ V_1=  <\frac{\partial }{\partial \th_1}, ..., \frac{\partial }{\partial \th_n} > /< \nabla I_{-1}|_{\bar {\q}}>, \ V_2=<\frac{\partial }{\partial \vp_1}, ..., \frac{\partial }{\partial \vp_n} > .\]
where  $ <b_1 ..., b_n >$ is the linear space spanned by the vectors $b_1, ..., b_n$. 

   We claim that $\mathcal{H}(f, \bar  {\q})$ is block-diagonal with respect to this splitting of 
   $T_{\bar {\q}} S^-_c$. Direct computation leads to  $\frac{\partial ^2 I_{-1}}{\partial \th_i \partial \vp_j} |_{\bar {\q}}=0$ for all pairs of $(i, j)$.  We claim that $\frac{\partial ^2 U_{-1}}{\partial \th_i \partial \vp_j} |_{\bar {\q}}=0$ for all pairs of $(i, j)$.  Denote by $\q +h_i$ the coordinate $(\th_1, ..., \th_i+h, ...,  \th_n, \vp_1,  ..., \vp_n)$, by  $\q +k_j$ the coordinate $(\th_1, ...,  \th_n, \vp_1,  ..., \vp_j+k, ..., \vp_n)$. Then 
  \begin{align*}
  &\frac{\partial ^2 U_{-1}}{\partial \vp_j \partial \th_i } |_{\bar {\q}}=\lim _{k \to 0} \frac{1}{2k}( \frac{\partial  U_{-1}}{\partial \th_i } |_{\bar {\q}+k_j}    -  \frac{\partial  U_{-1}}{\partial \th_i  } |_{\bar {\q} -k_j}   )\\
  &=\lim _{(h,k)\to (0,0)} \frac{1}{2hk}  (  U_{-1}(\bar {\q}+k_j + h_i)  - U_{-1}(\bar {\q}+k_j) 
- U_{-1}(\bar {\q}-k_j + h_i)  + U_{-1}(\bar {\q}-k_j) )\\
&=0. 
  \end{align*}
  Here, we use the symmetry $U_{-1}(\bar  \q +k_j)= U_{-1}(\bar  \q -k_j)$,  $U_{-1}(\bar {\q}+k_j + h_i) =
   U_{-1}(\bar {\q}-k_j + h_i)$.  Thus, $\mathcal{H}(f, \bar  {\q})$ is block-diagonal, 
   $$\mathcal{H}(f, \bar  {\q}) |_{\bar  \q}=  D^2 U_{-1}-\bar  \lambda D^2I_{-1}= \diag\{ [\frac{\partial^2 U}{\partial \th_i\partial\th_j}-\bar  \lambda   \frac{\partial^2 I_{-1}}{\partial \th_i\partial\th_j}],  [\frac{\partial^2 U}{\partial \vp_i\partial\vp_j}-\bar  \lambda   \frac{\partial^2 I_{-1}}{\partial \vp_i\partial\vp_j}]  \}. $$  
   Denote by $\mathcal{H}_1,  \mathcal{H}_2$ the two $n\times n$ blocks respectively.  If $b \in T_{\bar {\q}} S^-_c$ has decomposition $b=b_1+b_2$, then 
\begin{equation} \label{equ:Hes}
\mathcal{H}(f, \bar  {\q})(b) = \mathcal{H}_1(b_1) +   \mathcal{H}_2(b_2).  \end{equation}

We pass  to find the inertia of $\mathcal H_1$ on $V_1$ and that of  $\mathcal H_2$ on $V_2$.  The first one has already been found in the above proof. Indeed, it is  the Hessian of $U_{-1}|_{S^-_c\cap (\H^1)^n}$ 
So $\mathcal H_1$ on $V_1$ has inertia $ (n_0, n_+, n_-)= (0, n-1, 0)$.

\begin{proof}[Proof of Theorem \ref{thm:cc_h1&ind}, Part II]
	    By the above argument, it only remains to show that  $\mathcal H_2$ on $V_2$ has inertia $ (n_0, n_+, n_-)= (1, 1, n-2)$.

    Direct computation shows that $\mathcal H_2$  is 
     \begin{align*}
     & \begin{bmatrix}
     \sum\limits_{\substack{j=1, j\ne 1}}^n \frac{-m_1m_j\cosh \bar  \th_1\cosh \bar  \th_j}{\sinh^3d_{1j}} & \frac{m_1m_2 \cosh \bar  \th_1\cosh \bar  \th_2}{\sinh^3d_{12}}& \cdots& \frac{m_1m_n \cosh \bar  \th_1\cosh \bar  \th_n}{\sinh^3d_{1n}}\\
     \frac{m_2m_1\cosh \bar  \th_1\cosh \bar  \th_2}{\sinh^3d_{12}}&  \sum\limits_{\substack{j=1,j\ne 2}}^n \frac{-m_2m_j\cosh \bar  \th_2\cosh \bar  \th_j}{\sinh^3d_{2j}}&\cdots& \frac{m_2m_n\cosh \bar  \th_2\cosh \bar  \th_n}{\sinh^3d_{2n}}\\
     \cdots&  \cdots& \cdots& \cdots\displaybreak[0]\\
     \frac{m_1m_n\cosh \bar  \th_1\cosh \bar  \th_n}{\sinh^3d_{1n}}&\cdots&\cdots&\sum\limits_{\substack{j=1, j\ne n}}^n \frac{-m_nm_j\cosh \bar  \th_n\cosh \bar  \th_j}{\sinh^3d_{nj}}
     \end{bmatrix}\\
     &-2\bar  \lambda \begin{bmatrix}
     m_1 \cosh^2\bar  \th_1&0&\cdots&0\\
     0&m_2 \cosh^2\bar  \th_2&\cdots&0\\
     \cdots&  \cdots& \cdots& \cdots\\
     0&\cdots &\cdots &m_n \cosh^2\bar  \th_n\
     \end{bmatrix}. 
     \end{align*}

     Let 
     $C:={\rm diag}\{ \cosh \bar  \th_1, \cdots, \cosh \bar  \th_n \},
     	M :={\rm diag}\{ m_1, \cdots, m_n \}, $ and 
     		\begin{align*}
     	A:&=\begin{bmatrix}
     	\sum\limits_{\substack{j=1, j\ne 1}}^n \frac{-m_j\cosh \bar  \th_j}{\cosh \bar  \th_1\sinh^3d_{1j}} & \frac{m_2}{\sinh^3d_{12}}& \cdots& \frac{m_n}{\sinh^3d_{1n}}\\
     	\frac{m_1}{\sinh^3d_{12}}&  \sum\limits_{\substack{j=1,j\ne 2}}^n \frac{-m_j\cosh \bar  \th_j}{\cosh \bar  \th_2\sinh^3d_{2j}}&\cdots& \frac{m_n}{\sinh^3d_{2n}}\\
     	\cdots&  \cdots& \cdots& \cdots\displaybreak[0]\\
     	\frac{m_1}{\sinh^3d_{1n}}&\cdots&\cdots&\sum\limits_{\substack{j=1, j\ne n}}^n \frac{-m_j\cosh \bar  \th_j}{\cosh \bar  \th_n\sinh^3d_{nj}}
     	\end{bmatrix}.
     	\end{align*}
     	Then it is easy to check that $\mathcal H_2 =CM (A-2\bar  \lambda)C$.  Let $D=CM$.  Note that  $CM$, $C, D$  are all positive definite and diagonal.   Since
     \begin{align*}
    \mathcal H_2 &=[(D ^{\frac{1}{2}})^TD ^{\frac{1}{2}} (A-2\bar  \lambda) D ^{\frac{-1}{2}}D ^{\frac{1}{2}}]C \\
    &=C^{\frac{1}{2}}C^{\frac{-1}{2}}  [(D ^{\frac{1}{2}})^TD ^{\frac{1}{2}} (A-2\bar  \lambda) D ^{\frac{-1}{2}}D ^{\frac{1}{2}} ] C^{\frac{1}{2}}(C^{\frac{1}{2}})^T. 
     \end{align*}
     	 the inertia of $\mathcal H_2$ equals that of $A-2\bar  \lambda$ by  Sylvester's law of inertia.  
  
     	  For 	the inertia of  $A-2\bar  \lambda$, we will  study the eigenvalues of $A$ and compare them with the negative constant $2\bar  \lambda$. First,  there are two obvious eigenvectors of $A$:
     	\begin{align*}
     	\c_1&=(\cosh \bar  \th_1, \cdots, \cosh \bar  \th_n), \ \ & A\c_1&=0\c_1,\\
     	\c_2&=(\sinh \bar  \th_1, \cdots, \sinh \bar  \th_n), \ \ & A\c_2&=2\bar  \lambda \c_2. \end{align*}
     	The first vector  $\c_1$ can be obtained by inspecting the matrix, and the second vector $\c_2$ can be seen from equation \eqref{equ:reeh1}. 
     
     		Note that  the matrix $CMAC$ is symmetric,  sum of each row is zero, the diagonal is negative and other elements are positive.  By the argument used in the proof of Theorem \ref{thm:cc_h1&ind} (part I), we conclude that all other eigenvalues of $A$ is strictly negative.

     	We  claim that all other eigenvalues of $A$ are smaller than $2\bar  \lambda$. 
     	The idea is to consider the linear vector field $Y=A\u$  in $\R^n, \u=(u_1, \cdots, u_n)^T\in \R^n.$ 
Introduce  an inner product  in $\R^n$, $(\u, \v)=\u^T M \v$.
     	Then $A$ is symmetric with respect to this inner product and all eigenvectors of $A$ are mutually orthogonal. 
Then the line $t\c_1$ consists all fixed points of the flow, and the $(n-1)$-dimensional space $\c_1^\bot$ is invariant under the flow. In particular, each of  the other eigenvectors corresponds to  a stable manifold of the flow. 
     	Conley observed that to show that all other eigenvalues of $A$ are smaller than $2\bar  \lambda$ is equivalent to showing that  the line  $t\c_2$ attracts the flow lines. It is enough to find a cone $K\in \c_1^\bot $ around $t\c_2$ that is carried strictly inside itself by the flow.
     	\begin{figure}[!h]
     			\includegraphics [width=0.4 \textwidth] {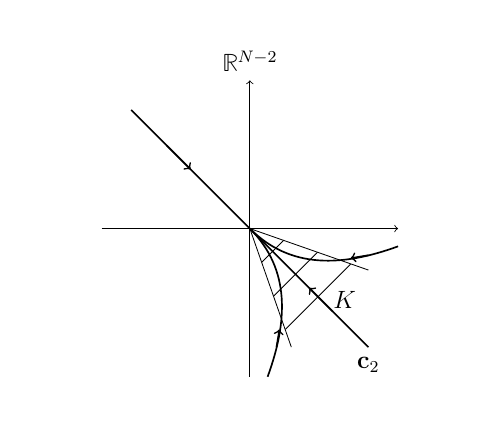}
     		\caption{The linear flow in  $\c_1^\bot$}
     		\label{fig:conley}
     	\end{figure} 
   Define the polyhedral cone  in $\c_1^\bot $ as 
     	\[ K=\left\{ \u\in \R^n| \sum_{i=1} ^n m_i \cosh \bar  \th_i u_i=0, \frac{u_1}{\cosh \bar  \th_1}\le \frac{u_2}{\cosh \bar  \th_2}\le \cdots \le \frac{u_n}{\cosh \bar  \th_n} \right \}. \]
   We verify that  $\c_2 \in K$.  First, $\c_2 \bot \c_1$ since $\c_2$ is also an eigenvector of $A$.  Or,  the identity $ \sum_{i=1}^n m_i  x_i   w_i=0$ 
    of Theorem \ref{thm:mc} reads $(\c_1,  \c_2)=0$ in this case. 
 Second, the inequalities in  the definition read 
     $ \frac{\sinh \bar  \th_1}{\cosh \bar  \th_1}\le \frac{\sinh \bar  \th_2}{\cosh \bar  \th_2}\le \cdots \le \frac{\sinh \bar  \th_n}{\cosh \bar  \th_n}$, which is true since 
     $\bar  \th_1< \bar  \th_2<\cdots<\bar  \th_n$. 
 
 We  verify that $\c_2$  is the only  eigenvector of $A$ in $K$.  Since $\bar  \th_1< \bar  \th_2<\cdots<\bar  \th_n$ and $\sum_{i=1} ^n m_i \sinh 2\bar  \th_i =0$, we  assume that there is some $I\in (1, n)$ such that 
 $\bar \th_1 < ...<\bar \th_I\le 0<\bar \th_{I+1}<...<\bar \th_n.$
 Then by  the definition of $K$, for any $\u=(u_1, ..., u_n)^T\in K$, we have
 \begin{equation*}
 \begin{cases}
\frac{u_i}{\cosh \bar  \th_i} m_i \sinh 2\bar \th_i \ge   \frac{u_I}{\cosh \bar  \th_I} m_i \sinh 2\bar \th_i, \  & 1\le i\le I; \cr
\frac{u_i}{\cosh \bar  \th_i} m_i \sinh 2\bar \th_i \ge   \frac{u_{I+1}}{\cosh \bar  \th_{I+1}} m_i \sinh 2\bar \th_i, \ & I+1\le i\le n.  
 \end{cases}
 \end{equation*}
Then for any point $\u\in K$, we have 
\begin{align*}
 2(\u,  \c_2)& =   2 \sum_{i=1}^n u_i m_i \sinh \bar \th_i =   \sum_{i=1}^n \frac{u_i}{\cosh \bar  \th_i} m_i \sinh 2 \bar\th_i   \\
 &=\sum_{i=1}^I \frac{u_i}{\cosh \bar  \th_i} m_i \sinh 2 \bar\th_i  +\sum_{i=I+1}^n \frac{u_i}{\cosh \bar  \th_i} m_i \sinh 2 \bar\th_i  \\
 &\ge  \frac{u_I}{\cosh \bar  \th_I}  \sum_{i=1}^I m_i \sinh 2 \bar\th_i  + \frac{u_{I+1}}{\cosh \bar  \th_{I+1}}  \sum_{i=I+1}^n m_i \sinh 2 \bar\th_i  \\
 &\ge\left(\frac{u_{I+1}}{\cosh \bar  \th_{I+1}}  -  \frac{u_I}{\cosh \bar  \th_I} \right)  \sum_{i=I+1}^n m_i \sinh 2 \bar\th_i  \\
 &\ge 0. 
\end{align*}
Moreover, the product is zero  only if $\frac{u_1}{\cosh \bar  \th_1}= \cdots = \frac{u_n}{\cosh \bar  \th_n}$, which is impossible. Hence, $(\u,  \c_2) >0$, and there is no other eigenvectors than $\c_2$ of $A$ in $K$.

     	The boundary $\partial K$ consists of points for which one or more equalities hold. However, except for the origin, at least one inequality must hold (otherwise $\u=t\c_1$). 
     	Consider a boundary point with 
     	\[ \frac{u_1}{\cosh \bar  \th_1}\le \cdots \frac{u_i}{\cosh \bar  \th_i}=\cdots = \frac{u_j}{\cosh \bar  \th_j}\le \cdots \le \frac{u_n}{\cosh \bar  \th_n}. \]
     	Let $g(\u)=\frac{ u_j}{\cosh \bar  \th_j}-\frac{ u_i}{\cosh \bar  \th_i}$. Then $g=0$
     	at this point, and $g$ is positive  in $K$.  
     	To prove that at this point  the flow is pointing inwards, see Figure \ref{fig:conley},  we show $L_Yg =\frac{ \dot{u}_j}{\cosh \bar  \th_j}-\frac{ \dot{u}_i}{\cosh \bar  \th_i}>0$. Direct computation shows that $\frac{ \dot{u}_j}{\cosh \bar  \th_j}-\frac{ \dot{u}_i}{\cosh \bar  \th_i}$ is
     	\begin{equation*} 
     	\begin{split}
     	&\sum _{k=1,k\ne j}^n \frac{m_k}{\cosh \bar  \th_j\sinh^3 d_{kj}}\left(u_k- \frac{u_j \cosh \bar  \th_k}{\cosh \bar  \th_j}\right)- \sum _{k=1,k\ne i}^n \frac{m_k}{\cosh \bar  \th_i\sinh^3 d_{ki}}\left(u_k- \frac{u_i \cosh \bar  \th_k}{\cosh \bar  \th_i}\right)\\
     	&= \sum _{k=1,k\ne i,j}^n   m_k\left(\frac{u_k}{\sinh^3 d_{kj} \cosh \bar  \th_j}- \frac{u_j \cosh \bar  \th_k}{\cosh^2 \bar  \th_j\sinh^3 d_{kj}} - \frac{u_k}{\sinh^3 d_{ki} \cosh \bar  \th_i}+ \frac{u_i \cosh \bar  \th_k}{\cosh^2 \bar  \th_i\sinh^3 d_{ki}}      \right) \\
     	& + \frac{m_i}{\sinh^3 d_{ij}\cosh \bar  \th_j}\left(u_i - \frac{u_j \cosh \bar  \th_i}{\cosh \bar  \th_j}\right)-\frac{m_j}{\sinh^3 d_{ij}\cosh \bar  \th_i}\left(u_j - \frac{u_i \cosh \bar  \th_j}{\cosh \bar  \th_i}\right).
     	\end{split}
     	\end{equation*}
     	Since $\frac{u_i}{\cosh \bar  \th_i} = \frac{u_j}{\cosh \bar  \th_j}$, the last two terms are zero, and the expression $ \frac{u_j \cosh \bar  \th_k}{\cosh^2 \bar  \th_j\sinh^3 d_{kj}}$ in the first part can be written as $\frac{u_i \cosh \bar  \th_k}{\cosh \bar  \th_i\cosh \bar  \th_j\sinh^3 d_{kj}}$. Then $L_Yg$ can be written as 
     	\[ \sum _{k=1,k\ne i,j}^n   m_k \left(  u_k-\frac{u_i \cosh \bar  \th_k}{ \cosh \bar  \th_i}  \right) \left(  \frac{1}{\sinh^3 d_{kj} \cosh \bar  \th_j}-\frac{1}{\sinh^3 d_{ki} \cosh \bar  \th_i} \right).\]
     	Every term in this expression  is non-negative by Proposition  \ref{prop:dis_ine_h1}. Proposition  \ref{prop:dis_ine_h1}  will be given after this proof:
     	 \begin{equation}\label{equ:h2ind}
     	\begin{cases}
     	  {\rm  If}\ k<i, \  {\rm  then} \   u_k-\frac{u_i \cosh \bar  \th_k}{ \cosh \bar  \th_i}  \le 0,   \frac{1}{\sinh^3 d_{kj} \cosh \bar  \th_j}-\frac{1}{\sinh^3 d_{ki} \cosh \bar  \th_i} <0.\cr
     {\rm  If}\ i\le k\le  j, \  {\rm  then} \   u_k-\frac{u_i \cosh \bar  \th_k}{ \cosh \bar  \th_i} =0. \cr 
      {\rm  If}\ j< k, \  {\rm  then} \ u_k-\frac{u_i \cosh \bar  \th_k}{ \cosh \bar  \th_i}  \ge 0,   \frac{1}{\sinh^3 d_{kj} \cosh \bar  \th_j}-\frac{1}{\sinh^3 d_{ki} \cosh \bar  \th_i} >0. 
     	\end{cases}
     	\end{equation}	
     	Moreover, at least one term is strictly positive since at least one inequality in the definition of the cone must hold.  Thus,  we have proved that on the boundary of the cone  the flow is pointing inwards, or, 
     all the other eigenvalues of $A$ are smaller than $2\bar  \lambda$.

      Hence,  the eigenvalues of $A-2\bar  \lambda$ are 
     	$-2\bar  \lambda>0, 0,  \lambda_3<0, \cdots,  \lambda_n<0.  $  Then, the inertia of $A-2\bar {\lambda }$ is $ (n_0, n_+, n_-)= (1, 1, n-2)$, so is the inertia of $\mathcal H_2$, on the space  $V_2$.   On the space $V_2/S^1$, obviously, the inertia is $ (n_0, n_+, n_-)= (0, 1, n-2)$. Combined with the inertia of $\mathcal H_1$, we conclude that the inertia of $U_{-1}|_{S^-_c/S^1}$ is  
     	 \[ (0, n-1, 0)+ (0, 1, n-2)=(0, n, n-2).  \]
    This completes the proof of Theorem \ref{thm:cc_h1&ind}.   
 \end{proof}

     	 \begin{proposition} \label{prop:dis_ine_h1}
     		If  $ \th_1< \th_2<\cdots <\th_n$,  then the following inequalities hold. 
     		\begin{enumerate}
     			\item If $k<i<j$, then $\sinh^3 (\th_j-\th_k) \cosh \th_j-\sinh^3 (\th_i-\th_k) \cosh \th_i>0$. 
     			\item If $i<j<k$, then $\sinh^3 (\th_k-\th_i)\cosh \th_i- \sinh^3 (\th_k-\th_j)\cosh \th_j>0$.
     		\end{enumerate}
     	\end{proposition}
     	\begin{proof}
Let
     		$ h(x)=\sinh^3 (x-\th_k) \cosh x-\sinh^3 (\th_i-\th_k) \cosh \th_i$ defined for $ x\ge \th_i.$
     		Then $h(\th_i)=0$, and
     		\begin{equation*} 
     		h'(x)
     		=\sinh^2 (x-\th_k)[ \cosh ( 2x-\th_k) +2\cosh (x-\th_k) \cosh x  ]>0.
     		\end{equation*}
     	This implies the first inequality. We omit the proof of the second one since it is similar.  
     	\end{proof}

  \subsection{ The  $\S^2$ case}    
  We omit the proof of the following result since it  is  similar to that of Proposition \ref{pro:h2lam}.

 \begin{proposition}\label{pro:s2lam}
 	Let $\q$ be an ordinary central configuration on  $\S^2$.  If $z_i>0$ for  all particles.  Then the multiplier  is negative. 
 \end{proposition}

We use an angle coordinate system for $\{(x,y,z)\in \S^2|   z>0 \}$,  $(\th , \vp)$,  with  $-\frac{\pi}{2}<\th<\frac{\pi}{2}, -\frac{\pi}{2}<\vp<\frac{\pi}{2}$.   The relationship between Cartesian coordinates $(x,y,z)$ and $(\th , \vp)$ is   
\[  (x,y,z)=(\sin \th, \cos \th \sin\vp,\cos \th \cos \vp).\] 
Then  $\S^1=\S^1_{xz}$ is parameterized by $(\th, 0)$ and  $(\S^2)^n$ is parameterized by  $(\th_1, ..., \th_n, \vp_1, ..., \vp_n)$, and the momentum of inertia is 
 \[ I_1(\q) =    \sum m_i( \sin^2 \th+ \cos^2 \th \sin\vp ).      \] 
 
 \begin{proof}[Proof of Theorem \ref{thm:cc_s1&ind}]
 	 For the number of geodesic  ordinary central configurations,  we apply  the argument used in the proof of Theorem \ref{thm:cc_h1&ind}.   All arguments run well except showing  that the Hessian of $U_1$ restricted on  $\mathcal M_c\cap (\S^1)^n$
  at each critical point  is positive definite. By direct computation, we obtain  the Hessian 
\begin{align*}
 & 2\begin{bmatrix}
 	\sum\limits_{\substack{j=1, j\ne 1}}^n \frac{m_1m_j\cos d_{1j}}{\sin ^3d_{1j}} & -\frac{m_1m_2\cos d_{12}}{\sin ^3d_{12}}& \cdots& -\frac{m_1m_n\cos d_{1n}}{\sin ^3d_{1n}}\\
 	-\frac{m_2m_1\cos d_{12}}{\sin ^3d_{12}}&  \sum\limits_{\substack{j=1,j\ne 2}}^n \frac{m_2m_j\cos d_{2j}}{\sin ^3d_{2j}}&\cdots& -\frac{m_2m_n\cos d_{2n}}{\sin ^3d_{2n}}\\
 	\cdots&  \cdots& \cdots& \cdots\displaybreak[0]\\
 	-\frac{m_1m_n\cos d_{1n}}{\sin ^3d_{1n}}&\cdots&\cdots&\sum\limits_{\substack{j=1, j\ne n}}^n \frac{m_nm_j\cos d_{nj}}{\sin ^3d_{nj}}
 \end{bmatrix}\\
 &-2\lambda \begin{bmatrix}
 	m_1 \cos 2\theta_1&0&\cdots&0\\
 	0&m_2 \cos 2\theta_2&\cdots&0\\
 	\cdots&  \cdots& \cdots& \cdots\\
 	0&\cdots &\cdots &m_n \cos 2\theta_n\
 \end{bmatrix}.
\end{align*}

Restricting  $c<\frac{1}{2} m_1$. Then   $m_i \sin^2 \th < \frac{1}{2} m_1$, 
$-\frac{\pi}{4} < \th_i <\frac{\pi}{4}$ for all $i$. Hence,  $\cos 2\th_i >0, d_{ij} <\frac{\pi}{2}$.  The second part  is positive definite since $\lambda <0$. 
Each  element of the first matrix not on the diagonal  is   negative.  By   the argument used in the proof of Theorem \ref{thm:cc_h1&ind} (part I), we  see that the first matrix is positive  semi-definite. Thus, the Hessian of  $U_1|_{\mathcal M_c\cap (\S^1)^n}$ is positive definite and  there are exactly $n!/2$ geodesic  ordinary central configurations  on $\mathcal M_c$ provided $c<\frac{m_1}{2}$.

  For the Morse index of  the geodesic  ordinary central configurations,  we need to  restrict further  $c<\frac{1}{4} m_1$, which leads to  $m_i \sin^2 \th < \frac{1}{4} m_1$, 
  $-\frac{\pi}{6} < \th_i <\frac{\pi}{6}$ for all $i$.  
   Then all the argument  used in the proof of Theorem \ref{thm:cc_h1&ind}  works if we  replace the  hyperbolic functions with the  trigonometrical ones. Especially, the inequalities \eqref{equ:h2ind} are replaced by the  following inequalities. 
  
   \begin{proposition} \label{prop:dis_ine_s1}
  	If  $-\frac{\pi}{6} < \th_1< \cdots <\th_n<\frac{\pi}{6}$, then   the following inequalities hold. 
  	\begin{enumerate}
  		\item If $k<i<j$, then $\sin^3 (\th_j-\th_k) \cos \th_j-\sin^3 (\th_i-\th_k) \cos \th_i>0$.
  		\item If $i<j<k$, then $\sin^3 (\th_k-\th_i)\cos \th_i- \sin^3 (\th_k-\th_j)\cos \th_j>0$.
  	\end{enumerate}
  \end{proposition}
  \begin{proof}	
  We only prove the first inequality. Let
  	$ h(x)=\sin^3 (x-\th_k) \cos x-\sin^3 (\th_i-\th_k) \cos \th_i$ defined for $ x\ge \th_i.$
  	Then $h(\th_i)=0$, and
  	\begin{equation*} 
  	h'(x)
  =\sin^2 (x-\th_k)[ \cos ( 2x-\th_k) +2\cos (x-\th_k) \cos x  ]>0.
  	\end{equation*}
  	This proves the first inequality. 
  \end{proof} 
  
 This completes the proof of Theorem \ref{thm:cc_s1&ind}.   
 \end{proof}

\section*{Appendix: The  relative  equilibria and  central configurations}

Recall that $\S_{xy}^1:=\{(x,y,z,w) \in \S^3: z=w=0\}$, $\S_{zw}^1:=\{(x,y,z,w) \in \S^3: x=y=0\}$. Recall that the critical points of $U_1$ are \emph{special central configurations}, the critical points of $U_1-\lambda I_1, U_{-1}-\lambda I_{-1}$ that are not special central configurations  are  \emph{ordinary  central configurations}. The motions in the form of $\q(t)=\q(0)$ are called \emph{equilibria}, and the motions in the form of $\exp(t\xi)\q$  ($\xi\ne 0$) are called  \emph{relative  equilibria}.

		\begin{proposition}[\cite{DSZ18}] \label{pro:re_cc}
			\begin{itemize}
				\item 	Let $\q=(\q_1,\dots,\q_n)\in (\H^3)^n$ be 
				an ordinary central configuration  with multiplier $\lambda$. Then  the associated  relative equilibria are 	
				$B_{\alpha, \beta}(t)\q$ with $\a=\sqrt{-2\lambda}\cos s, \b=\sqrt{-2\lambda}\sin s, s\in (0, 2\pi]$.
		\item 		Let $\q=(\q_1,\dots,\q_n)\in (\S^3)^n$ be a special central configuration.  Then  the associated equilibrium  is $\q(t)=\q$. The associated 
		relative equilibria are 
			\begin{itemize} 
			\item
		$A_{\alpha, \beta}(t)\q$ with $\a,\b \in \R$ if   all particles are on  $\S^1_{xy}\cup \S^1_{zw}$; 
		\item $A_{\alpha, \beta}(t)\q$ with $\b=\pm \a, \a\in \R$, if  not all particles are on $\S^1_{xy}\cup \S^1_{zw}$.
			\end{itemize}
		\item 	Let $\q=(\q_1,\dots,\q_n)\in (\S^3)^n$ be an ordinary central configuration.  The associated 
		relative equilibria are
			\begin{itemize} 
				\item  $A_{\alpha, \beta}(t)\q$ with $\a=\sqrt{2\lambda}\sinh s, \b=\sqrt{2\lambda}\cosh s,$   $s\in \R$, if  $\lambda >0$;
				\item  $A_{\alpha, \beta}(t)\q$ with $\a=\sqrt{-2\lambda}\cosh s, \b=\sqrt{-2\lambda}\sinh s$,   $ s\in \R$,  if  $\lambda <0$. 
			\end{itemize}
			\end{itemize}
		\end{proposition}  
		
		For one ordinary central configuration, among  the set of associated relative  equilibria, there are  periodic ones (in $H^3$,   $\b=0$;  in  $S^3$, $\a +k\b=0$ for some $k\in \Z$), and quasi-periodic ones (in $H^3$, none; in  $S^3$, $\a +k\b\ne0$ for any $k\in \Z$).

		In  Newtonian $n$-body problem, the  relative equilibrium is always planar, while in curved $n$-body problem, the set of  relative equilibria has richer structure. We  divide them into three classes, the geodesic ones, the 2-dimensional   ones,  and the 3-dimensional  ones. 	 A \emph{geodesic relative  equilibrium} is one with all particles on the same geodesic for all $t$; a \emph{2-dimensional  relative  equilibrium} is one  with all particles on the same 2-dimensional  great sphere for all $t$ but not on a same geodesic;   the others  are  \emph{3-dimensional   relative equilibria}. 
		
		If a  $k$-dimensional  relative  equilibrium  is associated with a $m$-dimensional 
		 configuration, then $k\ge m$. 
		Let $Q(t)\q$ be one relative  equilibrium. Then it is a geodesic one if  $\q$ is on a   geodesic  and $Q(t)$ keeps that geodesic;  it is a  2-dimensional one   if $\q$ is on a  2-dimensional  great sphere ($\q$ may be a geodesic one in that sphere),  
		and $Q(t)$ keeps that 2-dimensional  great sphere.

	Let $G_1=SO(2)\times SO^+(1,1)$ and $G_2=SO(2)\times SO(2)$.   
		 Assume that  $\q$ is an ordinary central configuration in $\H^3$ (resp. $\S^3$). 
		Then $g\q$ is an ordinary central configuration with the same multiplier if $g$ is in $G_1$ (resp. $G_2$).  If $B_{\a, \b}(t)\q$ (resp. $A_{\a, \b}(t)\q$) are   relative equilibria associated with $\q$, then the  relative equilibria associated with $g\q$ are 
		\[  B_{\a, \b}(t)g\q=gB_{\a, \b}(t)\q, \ ({\rm resp.} \  A_{\a, \b}(t)g\q=gA_{\a, \b}(t)\q).    \]	
	Let $\tau$ be the isometry in $O(4)$, $\tau (x, y, z,w)=(z, w, x, y)$. By Theorem \ref{thm:s2cc}, if   $\q$ is an ordinary central configuration in $\S^3$  with  multiplier $\lambda$, then the multiplier of $\tau \q$ is $-\lambda$.  If $A_{\a, \b}(t)\q$ are   relative equilibria associated with $\q$, then the   solutions associated with $\tau\q$ are
		\[    A_{\b, \a}(t)\tau\q=\tau A_{\a, \b}(t)\q.     \]
	
		Thus,  thanks to  Theorem \ref{thm:s2cc} and Theorem \ref{thm:h2cc}, to find  geodesic and   2-dimensional   relative equilibria,  it is enough to assume that the  associated  ordinary central configuration lies on $\H^2_{xyw}$ for the $\H^3$ case, and  on  $\S^2_{xyz}$ with negative multiplier for the $\S^3$ case.

		\begin{proposition}
			\label{pro:h3re} 
			Consider  the curved $n$-body problem in $\H^3$. Let $G_1=SO(2)\times SO^+(1,1)$. 
			\begin{itemize}
				\item 	There is no geodesic relative equilibrium.
				\item Any 2-dimensional   relative equilibria must be in one of the following  three  forms: 
				\begin{itemize}
					\item $gB_{\pm\sqrt{-2\lambda}, 0}(t)\q$ and $gB_{0, \pm\sqrt{-2\lambda}}(t)\q$ for $\q$ being a geodesic ordinary central configuration on $\H^1_{xw}$ with multiplier $\lambda$;
					\item $gB_{\pm\sqrt{-2\lambda}, 0}(t)\q$  for $\q$ being a 2-dimensional  ordinary central configuration on $\H^2_{xyw}$ with multiplier $\lambda$,
				\end{itemize}
				where $g$  is some isometry in $ G_1$.  
				\item Any  other  relative equilibrium is 3-dimensional . 
			\end{itemize}
		\end{proposition}
		
	\begin{proof}
	Let  $\q$ be  an ordinary central configuration on $\H^1_{xw}$ with multiplier $\lambda$. The 1-parameter  subgroup $B_{\a, \b}(t)$ keeps the geodesic $\H^1_{xw}$ only if $\a=\b=0$, 
	which is impossible since $\lambda<0$ by Proposition \ref{pro:h2lam} and $2\lambda =-(\alpha^2+\beta^2)$.
	So there is no geodesic relative equilibrium.  Obviously, the 1-parameter  subgroup $B_{\a, \b}(t)$ keeps a 2-dimensional great  sphere containing  $\H^1_{xw}$ only if $\a=0$ or $\b=0$.  If $\b=0$ (resp. $\a=0$), the associated 2-dimensional   relative equilibrium is   $B_{\pm\sqrt{-2\lambda}, 0}(t)\q$ (resp. $B_{0, \pm\sqrt{-2\lambda}}(t)\q$).
	
	Let  $\q$ be  a 2-dimensional ordinary central configuration on $\H^2_{xyw}$.
	Obviously,  the 1-parameter  subgroup $B_{\a, \b}(t)$ keeps  $\H^2_{xyw}$   only if $\b=0$.  So  2-dimensional   relative equilibria associated with $\q$ is  $B_{\pm\sqrt{-2\lambda}, 0}(t)\q$.  By the discussion before Proposition \ref{pro:h3re}, the proof is complete. 
\end{proof}

	Recall that a relative equilibrium  $B_{0, \b}(t)\q$ is  \emph{hyperbolic}. In  \cite{PS19}, P\'{e}rez-Chavela and  S\'{a}nchez-Cerritos 	consider  2-dimensional  hyperbolic relative equilibria.   They  show that if the masses are equal, the configuration of  such   relative equilibria could not be a regular polygon. In fact, those motions can be characterized as follows: 
	\begin{proposition}
	All 2-dimensional  hyperbolic relative  equilibria must be associated with  geodesic ordinary central configurations.  Given $n$ masses, there are exactly $n!$ families of  2-dimensional  hyperbolic relative equilibria, one for each ordering of  the masses along the geodesic.   
	\end{proposition}
	\begin{proof}
	The first part is from the second statement of Proposition \ref{pro:h3re}. We  can  prove it   directly. Since the motion is 2-dimensional, we use the  Poincar\'e half plane model:  $H, (x,y), y>0, ds^2 =  \frac{dx^2 +dy^2 }{y^2}$. Then the kinetic is $K=\frac{1}{2}\sum m_i\frac{\dot x_i^2 +\dot y_i^2 }{y_i^2}$. In this model, the hyperbolic 1-parameter subgroup acts on $H$ by 
$(x, y)\mapsto e^{\a s}(x, y)$ \cite{DPV12}. Thus, the vector filed on $H^n$ generated by the  hyperbolic 1-parameter subgroup is $\xi_{H^n}(\q)=\a(x_1, y_1, ..., x_n, y_n)$. Then the augmented potential is 
\[  U+ K(\xi_{H^n}(\q))=U+\frac{\a^2 }{2}\sum_{i=1}^n m_i \frac{x_i^2 }{y_i^2} +\frac{\a^2 }{2} \sum_{i=1}^n m_i. \]
	If $\q(t)=e^{\a t} \q(0)$ is a hyperbolic relative equilibrium on $H$, then $\q(0)$ is a  critical point of the above augmented potential. That is, $\q(0)$ must satisfy the equation 
\begin{equation}\label{equ:h2rehyperb}
\nabla_{\q_i} U  = -\frac{\a^2 }{2}\nabla_{\q_i}  \sum_{i=1}^n m_i \frac{x_i^2 }{y_i^2}  =-\a^2 m_i \frac{1 }{y_i} ( x_i y_i, -x_i^2). 
\end{equation}     
We claim  that the critical points of this potential must geodesic configurations. Recall that the geodesics on $H$ are straight lines  and circles perpendicular to the $x$-axis. 
	Assume that the particles of $\q(0)$ are distributed on several circular geodesics $x^2+y^2 =R_j^2, j=1, ..., p$ and that $R_j\le R_p$ if $1\le j\le p$. Consider the equation \eqref{equ:h2rehyperb} for one particle, say $\q_1$,  on largest circle. Note that  the  right side of \eqref{equ:h2rehyperb}  is a vector  tangent to
	the largest circle, but the left side, the force exerted on $\q_1$,   is pointing inwards since there are particles on some smaller circle. This contradiction shows that  the critical points of the augmented potential has to be    geodesic configurations. 
	
	The second part is from Theorem \ref{thm:cc_h1&ind}. 
	\end{proof}

	Note that for all 2-dimensional  hyperbolic relative  equilibria, the velocities are orthogonal to the geodesic containing the configuration. This is true for any relative equilibria associated a  geodesic configuration.  By Proposition \ref{pro:h3re}, we may assume that the geodesic is $\H^1_{xw}$.  Then the velocity of the $i$-th particle  is $(0, \alpha x_i, \beta w_i, 0)$ for some $\a, \b$, which is orthogonal to the geodesic.  
		
		\begin{proposition} \label{pro:s3rec}
			For the curved $n$-body problem in $\S^3$, consider the  relative equilibria  associated with ordinary central configurations.  Let $\tau $ be the isometry $\tau (x, y, z,w)=(z, w, x, y)$ and $G_2=SO(2)\times SO(2)$.  
			\begin{itemize}
				\item 	There is no geodesic relative equilibrium.
				\item Any 2-dimensional   relative equilibria must be in one of the following  four   forms: 
				\begin{itemize}
					\item $gA_{\pm\sqrt{-2\lambda}, 0}(t)\q$ and $\tau gA_{\pm\sqrt{-2\lambda}, 0}(t)\q$ for $\q$ being a geodesic ordinary central configuration on $\S^1_{xz}$ with multiplier $\lambda<0$.
					\item $gA_{\pm\sqrt{-2\lambda}, 0}(t)\q$ and $\tau gA_{\pm\sqrt{-2\lambda}, 0}(t)\q$ for $\q$ being a 2-dimensional  ordinary central configuration on $\S^2_{xyz}$ with multiplier $\lambda<0$.
				\end{itemize}
				where $g$  is some isometry in $G_2$.  
				\item Any  other  relative equilibrium is 3-dimensional . 
			\end{itemize}
		\end{proposition}
		We omit the proof   since it is similar to that of Proposition \ref{pro:h3re}. Note that the ordinary central configurations on $\S^2_{xyz}$ with multiplier $\lambda>0$ lead to 3-dimensional relative  equilibria.  Similar to the hyperbolic case, the velocity of  any relative  equilibrium associated a  geodesic configuration is orthogonal to the geodesic containing the configuration.

		We now turn to the special central configurations. 	 In this case, the velocity of   relative equilibrium associated a  geodesic configuration  can be along or  not along  the geodesic containing the configuration.  There is no need to discuss the dimension of equilibrium solutions. 
		\begin{proposition} 	\label{pro:s3ec}
			For the curved $n$-body problem in $\S^3$, consider the  relative equilibria associated with special central configurations. 
			\begin{itemize}
				\item Any geodesic   relative equilibrium  must be  $A_{\a, 0}(t)\q$, or $\tau A_{\a, 0}(t)\q,  \a\ne \R$  for $\q$ being a special central configuration   on $\S^1_{xy}$.  
				\item  There is no 2-dimensional  relative equilibrium. 
				\item Any  other  relative equilibrium  is 3-dimensional. 
			\end{itemize}
		\end{proposition}

		\begin{proof}
			The symmetry group for special central configurations  is $O(4)$. 
			Any geodesic special central configuration is in the form of $g\q$, where $\q$ is on $\S^1_{xy}$ and $g\in O(4)$.  If   $g\q$ is  on $\S^1_{xy}$  (resp. $\S^1_{zw}$), 
			by Proposition \ref{pro:re_cc}, the associated  relative equilibria  are $A_{\a, \b}(t)\q$ for any $\a, \b\in \R$, which are  the same as $A_{\a, 0}(t)\q$ (resp.  $\tau A_{\a, 0}(t)\q$). 
			If   $g\q$ is not  on $\S^1_{xy}$ nor on $\S^1_{zw}$, by Proposition \ref{pro:re_cc}, the associated  relative equilibria  are $A_{\a, \pm \a}(t)\q$, $\a\in \R$. The  relative equilibrium is geodesic only if  $\a=0$, and it is not geodesic nor 2-dimensional  otherwise.

			Let $\q$ be a 2-dimensional special central configuration. We claim that $\q$  can not be within $\S^1_{xy}\bigcup \S^1_{zw}$. Note that the particles on one of the two circles must be collinear since the configuration  is contained in a 3-dimensional hyperplane. Assume that the particles on $\S^1_{xy}$ is collinear. The number of particles on $\S^1_{xy}$ is one otherwise $\q\in \D_+$. Then all particles of $\q$  is  within one 2-dimensional hemisphere, which is impossible, see Section 12.3  of  \cite{Dia13-1} or \cite{Zhu20-2}. The contradiction proves the claim.  Hence, the associated  relative equilibria must be $A_{\a, \pm \a}(t)\q$,  $\a\in \R$.  The  relative equilibrium is 2-dimensional  only if  $\a=0$. 
		\end{proof}

	Many  researches study the relative  equilibria of the curved $n$-body problem by restricting on a 2-dimensional physical space directly, i.e., on $T((\S^2)^n-\D_+)$ or $T((\H^2)^n-\D_-)$.  It seems more convenient to start with  the configurations, and to use the augmented potentials introduced in Theorem \ref{thm:recc} in some proper coordinates. Moreover, this restriction would make the configurations  counting clumsy. 


		Let us finish the discussion with one concrete example.  	In \cite{DS14}, Diacu and Sergiu consider 3-dimensional   relative equilibria of three-body  in $\S^3$ with the property:  The configuration  is not  geodesic and the three mutual distances are the same.  They show that the three masses must be equal. This result can be  obtained  quickly as follows: 
		
		The associated central configuration   must be  2-dimensional since there are only three bodies and the configuration is not  geodesic. It  can not be a special central configuration since it  is not  geodesic \cite{Zhu20-2}.
		 We may assume that it is on $\S^2_{xyz}$, i.e., $\q_i=(x_i, y_i, z_i)$. By Diacu and Zhu \cite{DZ20}, a three-body configuration on $\S^2_{xyz}$ is an ordinary central configuration  if and only if 
		\[  \sum m_i z_i x_i=\sum m_i z_i y_i =0, \ {\rm  } \    (z_1, z_2, z_3)=k(\sin^3 d_{23}, \sin^3 d_{13}, \sin^3 d_{12}).  \] 
		 Since the three mutual distances are the same,  we get immediately that the three masses are the same.

\section*{acknowledgment}
 I would like to thank Xiang Yu  and Ernesto P\'{e}rez-Chavela 
 for careful reading of the manuscript and helpful suggestions.   I  would  like to thank the referee for many helpful suggestions and pointing out a gap in the original manuscript. 
This work is  supported by NSFC(No.11801537)
and China Scholarship Council (CSC NO. 201806345013).

\end{document}